\documentclass[11pt,a4paper]{amsart}
\usepackage{mathrsfs}
\usepackage{amsfonts}
\usepackage{amssymb}
\usepackage{amsfonts, amscd, amsmath, mathrsfs, amssymb, amsthm, amsxtra, stmaryrd, bbding, epsfig, graphicx, latexsym, url, mathbbol, bbold,xcolor,comment}

\usepackage[papersize={8in,11in},textwidth=14.9cm,textheight=21cm,centering]{geometry}

\usepackage[textsize=small,color=yellow]{todonotes}

\usepackage[normalem]{ulem}
\newcommand\redsout{\bgroup\markoverwith{\textcolor{red}{\rule[0.5ex]{2pt}{0.4pt}}}\ULon}
\usepackage{soul}
\setstcolor{red}
\usepackage{cancel}

\usepackage{enumerate,enumitem}

\usepackage[colorlinks=true,citecolor=blue,linkcolor=blue]{hyperref}
\hypersetup{
	pdfstartpage=1,
	pdfstartview=FitH}

\DeclareFontFamily{U}{mathx}{\hyphenchar\font45}
\DeclareFontShape{U}{mathx}{m}{n}{
	<5> <6> <7> <8> <9> <10>
	<10.95> <12> <14.4> <17.28> <20.74> <24.88>
	mathx10
}{}
\DeclareSymbolFont{mathx}{U}{mathx}{m}{n}
\DeclareMathAccent{\widecheck}{\mathalpha}{mathx}{"71}

\usepackage{caption}
\numberwithin{equation}{section}

\allowdisplaybreaks

\newtheorem{theorem}{Theorem}[section]
\newtheorem{lemma}{Lemma}[section]

\newtheorem{proposition}{Proposition}[section]

\makeatletter
\newcounter{roem}
\renewcommand{\theroem}{\Roman{roem}}

\newcommand{\c@org@eq}{}
\let\c@org@eq\c@equation
\newcommand{\org@theeq}{}
\let\org@theeq\theequation

\newcommand{\setroem}{
	\let\c@equation\c@roem
	\let\theequation\theroem}

\newcommand{\setarab}{
	\let\c@equation\c@org@eq
	\let\theequation\org@theeq}
\makeatother

\newtheorem*{claim*}{Claim}

\theoremstyle{remark}
\newtheorem{remark}[theorem]{\bf Remark}

\newcommand{\ds}{\displaystyle}

\newcommand{\N}{\mathbb{N}}
\newcommand{\Q}{\mathbb{Q}}
\newcommand{\R}{\mathbb{R}}
\newcommand{\T}{\mathbb{T}}
\newcommand{\Z}{\mathbb{Z}}

\newcommand{\cF}{\mathcal{F}}

\newcommand{\cQ}{\mathcal{Q}}

\definecolor{brown}{RGB}{165,42,42}

\def\le{\leqslant}
\def\leq{\leqslant}
\def\ge{\geqslant}
\def\geq{\geqslant}

\usepackage{graphicx}
\usepackage{tikz}

\usepackage{tcolorbox}

\newcommand\blfootnote[1]{%
	\begingroup
	\renewcommand\thefootnote{}\footnote{#1}%
	\addtocounter{footnote}{-1}%
	\endgroup
}

\begin{document}
	\title[M{\"o}bius Disjointness Conjecture for a skew product]
	{M{\"o}bius Disjointness Conjecture for a skew product  \\
		on a circle and the Heisenberg nilmanifold}
	\author[Y.-K. Lau and J. Ma]{Yuk-Kam Lau  and  Jing Ma}
	
	\address{
		Weihai Institute for Interdisciplinary Research, Shandong University, China \mbox{\rm and} 
		Department of Mathematics, The University of Hong Kong, Pokfulam Road, Hong Kong}
	\email{yklau@maths.hku.hk}
	
	\address{School of Mathematics, Jilin University, Changchun 130012, PR China}
	\email{jma@jlu.edu.cn}
	\date{\today}
	
	
	\begin{abstract}
		We establish Sarnak's conjecture on M\"obius disjointness for the dynamical system of a skew product on a circle and the three-dimensional Heisenberg nilmanifold, first studied by Wen Huang, Jianya Liu and Ke Wang. We advance the work of Huang, Liu, Wang, and their followers to a broad generality by removing the previously imposed restrictive symmetry condition.
	\end{abstract}
	
	\blfootnote{Corresponding author: Jing Ma.}
	\maketitle
	
	\section{Introduction}
	
	Let $X$ be a compact metric space and $T: X\rightarrow X$ a continuous map. Following \cite{Sarnak2}, we refer the pair $(X, T)$ as a flow, which is also known as a topological dynamical system\footnote{Some authors require $T$ to be a homeomorphism, which will be described as an invertible topological dynamical system.}.  Sarnak \cite{Sarnak} formulated a M\"obius randomness principle which speculated that the M\"obius function $\mu(n)$ is linearly disjoint from any deterministic flow\footnote{A flow is deterministic if its topological entropy is zero.} $F$. This speculation is now a Sarnak's conjecture, also expressed as
	
	\noindent
	{\bf M\"obius disjointness conjecture}:
	\vskip2mm
	\centerline{$\ds 
		\lim_{N\rightarrow\infty}\frac 1N\sum_{n\leq N} \mu(n)f(T^{n}x)=0
		$, $\forall$
		$f\in C(X)$, $\forall$ $x\in X$.}  
	\noindent
	Here $C(X)$ denotes the space of all continuous functions on $X$.
	
	This conjecture of Sarnak has been established for certain cases, for example in Bourgain \cite{Bourgain2},  Bourgain-Sarnak-Ziegler \cite{BSZ}, Green-Tao \cite{GT}, Liu-Sarnak \cite{LS, LS2},
	Wang \cite{Wang}, Peckner \cite{Peckner}, Huang-Wang-Ye \cite{HWY}, Litman-Wang \cite{LW}, and Huang-Liu-Wang \cite{HLW}. Interested readers are referred to the survey paper \cite{Ferenczi}. As an attempt at further evidence, one may firstly study the linear disjointness between $\mu$ and distal flows because  distal flows are of zero-entropy, by \cite{Parry}. 
	A flow $(X, T)$  is called distal if
	\(
	\inf_{n\geq0}d(T^{n}x, T^{n}y)>0
	\)
	whenever $x\neq y$, where $d$ is a metric on $X$.   
	Furstenberg’s structure theorem of minimal distal flows \cite{Furstenberg2} asserts that skew products are building blocks of distal flows. Consequently one can further boil down the study to the case of distal flows given by skew products. 
	
	A skew product consists of a direct product of two spaces and an action. An example is the skew product $T= T_{\alpha,h}$ on the $2$-torus $\mathbb{T}^2=(\mathbb{R}/\mathbb{Z})^2$ given by
	\begin{equation*}
		T_{\alpha,h} :(x, y)\mapsto(x+\alpha, y+h(x)),
	\end{equation*}
	where $\alpha\in[0, 1)$ and $h: \mathbb{T}\rightarrow\mathbb{T}$ is a continuous function. 
	The flow $(\mathbb{T}^2, T_{\alpha,h})$ is distal, and its 
	M{\"o}bius disjointness  was firstly studied by Liu and Sarnak  in \cite{LS, LS2}. 
	They showed in \cite{LS} that the M{\"o}bius Disjointness Conjecture is true for $(\mathbb{T}^2, T_{\alpha,h})$ for any $\alpha$ and any analytic $h$ whose Fourier coefficients fulfil some conditions. 
	
	A salient point of their result, especially in the study of KAM theory, is its generality for all $\alpha$. 
	The required conditions for $h$ are a main concern, which has been pursued subsequently. 
	Wang \cite{Wang} showed the sufficiency of analytic $h$ (without extra conditions on the Fourier coefficients). Huang-Wang-Ye \cite{HWY}, Kanigowski-Lema\'{n}czyk-Radziwill \cite{KLR} and de Faveri \cite{Faveri} further weakened the condition of $h$ from  being analytic to $C^\infty$-smooth, $C^{2+\varepsilon}$-smooth, and  $C^{1+\varepsilon}$-smooth respectively.\footnote{For real $r>0$, a $C^r$-smooth periodic real-valued function $f$ on $\R$ is characterized by the estimate of its Fourier coefficients: $\widehat{f}(n)\ll_r n^{-r}$ for all $|n|\ge 1$. 
	}
	
	These recent works have brought in some notions of dynamical systems scholars to prove M\"obius disjointness for many cases, not restricted to the skew product on $\mathbb{T}^2$.  Suppose $T$ is a homeomorphism on the compact space $X$.  Let $M(X,T)$ be the set of all $T$-invariant Borel probability measures on $X$. The work in \cite{KLR} introduces a way to  derive M\"obius disjointness from the polynomial rate of rigidity of $(X,T,\nu)$ for all $\nu\in M(X,T)$. A few years earlier, in \cite{HWY}, it is shown that if $X$ is endowed a metric, a sub-polynomial measure complexity of $(X,T,\nu,d)$ for all $\nu\in M(X,T)$ implies the M\"obius disjointness of $(X,T)$. We shall give more details in Section~\ref{RM}.

	
	Recently, Huang-Liu-Wang \cite{HLW} considered the case of the skew product on $\mathbb{T}\times \Gamma\backslash G$, where $\Gamma\backslash G$ is the $3$-dimensional Heisenberg nilmanifold. 
	Explicitly, $\T$ is still the unit circle, and let 
	\begin{eqnarray}\label{GGamma}
		G=\begin{pmatrix}\begin{smallmatrix} 1 & \R& \R\\ & 1 & \R\\ & & 1\end{smallmatrix}\end{pmatrix} 
		\quad \mbox{ and } \quad 
		\Gamma =\begin{pmatrix}\begin{smallmatrix} 1 & \Z& \Z\\ & 1 & \Z \\ & & 1\end{smallmatrix}\end{pmatrix}.
	\end{eqnarray} 
	The skew product on $\mathbb{T}\times\Gamma\backslash G$ is given by the map  $S_\alpha:\mathbb{T}\times\Gamma\backslash G \to \mathbb{T}\times\Gamma\backslash G$,  
	\begin{equation}\label{S}
		S_\alpha :(t, \Gamma g)\mapsto
		\left(
		t +\alpha, \Gamma g
		\begin{pmatrix}
			\begin{smallmatrix}
				1 & \varphi(t) & \psi(t) \\
				0 & 1 & \eta(t) \\
				0 & 0 & 1    
			\end{smallmatrix}
		\end{pmatrix}
		\right),
	\end{equation}
	where $\varphi$, $\eta$ and $\psi$ are  periodic functions with period $1$, and $\alpha\in [0,1)$. We suppress the dependence of these functions in the notation of $S_\alpha$ for simplicity. Assume 
	\begin{enumerate}[leftmargin=20mm]
		\item[(i)] $\varphi$, $\eta$ and $\psi$ are $C^{\infty}$-smooth.
	\end{enumerate} 
	\noindent
	The  flow $(\mathbb{T}\times\Gamma\backslash G, S_\alpha)$, abbreviated as $({\rm X},S_\alpha)$, is  shown to be distal. Assume further the conditions
	\begin{enumerate}[leftmargin=20mm]
		\item[(ii)] $\varphi = \eta$,
		\item[(iii)] $\varphi$ and hence $\eta$ have zero mean, i.e. $\int_0^1 \varphi(t)\,\textup{d} t=\int_0^1 \eta(t)\,\textup{d} t=0$.
	\end{enumerate} 
	Huang et al. showed that for rational $\alpha$, the M\"obius disjointness conjecture for $({\rm X}, S_\alpha)$ holds, and for irrational $\alpha$, 
	the flow $({\rm X},S_\alpha)$ has sub-polynomial measure complexity  with respect to some natural metric and any measure in $M({\rm X}, S_\alpha)$. Consequently under the conditions (i)-(iii), the M\"obius disjointness conjecture for $({\rm X}, S_\alpha)$ holds for all $\alpha$. See Appendix~II and Theorem~1.1 of \cite{HLW}.
	
	Among the three conditions, condition (ii) is the most restrictive, as illustrated below: 
	$$
	\begin{pmatrix}
		\begin{smallmatrix}
			1 & a &  \\
			& 1 & b \\
			&  & 1    
		\end{smallmatrix}
	\end{pmatrix} 
	\begin{pmatrix}
		\begin{smallmatrix}
			1 & c &  \\
			& 1 & d \\
			&  & 1    
		\end{smallmatrix}
	\end{pmatrix}
	= 
	\begin{pmatrix}
		\begin{smallmatrix}
			1 & a+c & ad \\
			& 1 & b+d \\
			&  & 1    
		\end{smallmatrix}
	\end{pmatrix}
	\ \mbox{ while } \ 
	\begin{pmatrix}
		\begin{smallmatrix}
			1 & c &  \\
			& 1 & d \\
			&  & 1    
		\end{smallmatrix}
	\end{pmatrix} 
	\begin{pmatrix}
		\begin{smallmatrix}
			1 & a &  \\
			& 1 & b \\
			&  & 1    
		\end{smallmatrix}
	\end{pmatrix}
	= 
	\begin{pmatrix}
		\begin{smallmatrix}
			1 & a+c & bc \\
			& 1 & b+d \\
			&  & 1    
		\end{smallmatrix}
	\end{pmatrix}
	.
	$$
	The influence on the center, $ad$ and $bc$, generally differs. Imposing the conditions $a = b$ and $c = d$ restores symmetry and, to some extent, alleviates the non-commutativity inherent in the Heisenberg group. Motivated by the $\mathbb{T}^2$ case, it is natural to ask whether a weaker smoothness condition than (i) might be sufficient. These questions have been explored by He and Wang \cite{HW}, Ma and Wu \cite{MW}.
	
	In \cite{HW}, keeping $\eta=\varphi$ but relaxing $\varphi (=\eta)$ and $\psi$ to be $C^{2+2\varepsilon}$- and $C^{1+\varepsilon}$-smooth, the polynomial rate of rigidity for $({\rm X},S_\alpha)$ was established, hence the M\"obius disjointness. This work  
	combines the ideas from Huang-Liu-Wang \cite{HLW}, Kanigowski-Lema\'{n}czyk-Radziwill \cite{KLR} and de Faveri \cite{Faveri}. 
	Ma and Wu \cite{MW} followed \cite{HLW} more and {\it relaxed the condition (ii)} but kept the smoothness assumption  (i). In \cite[Theorem 1.1 and 1.2]{MW}, the M\"obius disjointness in the case of rational $\alpha$ was derived without the constraint $\eta=\varphi$ while the sub-polynomial measure complexity for irrational $\alpha$ can be established with the condition  $\eta=k\varphi$ in place of (ii), where $k$ is a (fixed) nonzero real number. 
	
	
	In this paper our primary goal is to {\it remove the condition (ii)} with a relaxation of (i) for the M\"obius disjointness for $(\mathbb{T}\times\Gamma\backslash G, S_{\alpha})$. This is Theorem~\ref{thm} below, (the case of irrational $\alpha$) which is a consequence of Theorem~\ref{intd}, because by Proposition~\ref{prop2.1}, the polynomial rate of rigidity (PR rigidity) is a sufficient condition for the M\"obius disjointness. 
	In addition to the PR rigidity, we can show  the sub-polynomial measure complexity without assuming condition (ii). 
	Precisely we get the following theorems.

	
	\begin{theorem}\label{intd}
		Let $\mathbb{T}$ be the unit circle and $\Gamma\backslash G$ the $3$-dimensional Heisenberg nilmanifold. 
		Let $\alpha\in [0, 1)$ and $S_\alpha$ be the skew product on $\mathbb{T}\times\Gamma\backslash G$ defined in \eqref{S}, where $\varphi, \eta, \psi$ are periodic functions from $\mathbb{R}$ to $\mathbb{R}$ with period $1$ and both $\varphi$ and $\eta$ have zero mean.
		
		Let $0<\varepsilon< 10^{-2}$. Assume that  $\psi$ is  $C^{1+\varepsilon}$-smooth and  $\varphi$, $\eta$ are $C^{2+2\varepsilon}$-smooth. 
		Then $(\T\times \Gamma\backslash G, S_\alpha,\nu)$ has PR rigidity for all $\nu\in M(\T\times \Gamma\backslash G, S_\alpha)$ and irrational $\alpha$.
	\end{theorem}
	
	\begin{theorem}\label{thm}
		Under the same setting and assumptions as in Theorem~\ref{intd},  the M\"obius function $\mu$ is  linearly disjoint from $(\mathbb{T}\times\Gamma\backslash G, S_\alpha)$ for all $\alpha$, i.e. Sarnak's M\"obius disjointness conjecture holds.
	\end{theorem}

	\begin{theorem}\label{prop}
		With the same setting as in Theorem~\ref{intd}, but now assuming that  $\varphi, \eta, \psi$ are $C^\infty$-smooth, the measure complexity of $(\mathbb{T}\times\Gamma\backslash G, S_\alpha, \rho)$ is sub-polynomial for any $\rho\in M(\mathbb{T}\times\Gamma\backslash G, S_\alpha)$ and irrational $\alpha$.
	\end{theorem}
	
	\begin{remark}
		The Sarnak conjecture for $(\mathbb{T}\times\Gamma\backslash G, S_\alpha)$ for irrational $\alpha$ follows from both Theorems~\ref{intd} and \ref{prop} but under different conditions of  smoothness. The required smoothness may have reached the argument's limit and the zero mean condition (iii) seems unavoidable in our applied methods.  We also prove Theorem~\ref{prop}, although it implies the Sarnak conjecture under the stronger smoothness condition, because at present there is no evidence that the sub-polynomial measure complexity follows from the PR rigidity. 
	\end{remark}
	
	Our proofs build on the work of Huang-Liu-Wang \cite{HWY} and de Faveri \cite{Faveri} (and the adaptation of He and Wang \cite{HW}). Many of the tools developed or applied in these papers are used here.  The key ingredient that plays the novelty lies in an effective use of the centre of the Heisenberg group and the delicate treatment of double exponential sums arising from the skew product. We end this section with a briefing on notation and organization.

	\noindent\textbf{Notation.}
	Here we present some commonly used notation.
	We write $e(x):=e^{2\pi ix}$ and  
	\(
	\|  x\| :=\min_{n\in\mathbb{Z}}|x-n|.
	\)
	For positive $A, $ the notations $B =O(A)$ or $B\ll A$ mean that there exists a positive
	constant $c$ such that $|B|\leq cA$.
	If the constant $c$ depends on a parameter $b$,
	we write $B =O_b (A)$ or $B\ll_b A$.
	The notation $A\asymp B$ means that $A\ll B$ and $B\ll A$.
	For a topological space $X, $ we use $C(X)$ to denote the set of all continuous complex-valued functions on $X$. We write \S\S~2.1 for Subsection 2.1, etc.
	
	\noindent\textbf{Organization.} Nine sections follow. Section 2 summarizes the basic ideas of polynomial rate of rigidity (PR rigidity) and measure complexity. Section 3 provides a preparation of some tools, including metrics on the $3$-dimensional Heisenberg nilmanifold and important results related to continued fractions. Sections 4-6 complete the proof of Theorem~\ref{intd}. Theorem~\ref{thm} is deduced from Theorem~\ref{intd} and Green-Tao's result for irrational and rational $\alpha$. This is done in Section 7. Sections 8-10 are devoted to the proof of Theorem~\ref{prop}.

	\section{Rigidity and  Measure complexity}\label{RM}
	
	In this section, we collect some concepts and facts from \cite{KLR} and \cite{HWY} without proof.
	
	Let $X$ be a compact space, $T:X\to X$ a homeomorphism and let $M(X, T)$ be the set of all $T$-invariant Borel probability measures on $X$.
	
	\subsection{PR rigidity} For any $\nu\in M(X,T)$, the system $(X,T,\nu)$ is said to be rigid if for any $g\in L^2(X,\nu)$, there exists a strictly increasing sequence $\{r_n\}_{n\ge 1}\subset \N$ such that 
	$$
	\|g\circ T^{r_n}-g\|_{L^2(\nu)}^2 \to 0 \quad \mbox{as $n\to \infty$}. 
	$$
	Moreover, $(X,T,\nu)$ is said to have polynomial rate of rigidity (PR rigidty) if there exists a linearly dense\footnote{A subset $\mathcal{F}$ of a space is linearly dense if the linear combinations of elements in $\mathcal{F}$ form  a dense subset.} subset $\mathcal{F}\subset C(X)$ such that for every $f\in \mathcal{F}$, there are $\delta>0$ and a strictly increasing sequence $\{r_n\}_{n\ge 1}\subset \N$ such that
	\begin{eqnarray}\label{PRr}
		\sum_{|j|\le r_n^\delta} \|f\circ T^{jr_n}-f\|_{L^2(\nu)}^2 \to 0 \quad \mbox{as $n\to \infty$}.     
	\end{eqnarray}

	The proposition below follows from Theorem~1.1 of \cite{KLR}.
	\begin{proposition}\label{prop2.1}
		If $(X,T, \nu)$ has PR rigidity for all $\nu\in M(X,T)$, then the M\"obius disjointness for $(X,T)$ holds.
	\end{proposition}
	
	\subsection{Measure complexity} Suppose the topology of $X$ is induced by a metric $d$. For $n\in\mathbb{N}$, we define
	\begin{equation}\label{bard}
		\bar{d}_n(x, y)=\frac 1n\sum_{j=0}^{n-1}d(T^{j}x, T^{j}y)
	\end{equation}
	for $x, y\in X$. For $\nu \in M(X,T)$, the measure complexity of $(X,T)$ with respect to $\nu$ (or  $(X,T,\nu)$) is concerned with the number of balls of radius $\varepsilon$ under the metric $\bar{d}_n$ that can cover the whole space $X$ except for a subset whose measure under $\nu$ is less than $\varepsilon$, for arbitrary $\varepsilon>0$. Below is the definition we need. 
	
	Let $\nu \in M(X,T)$ and $\varepsilon>0$.  Define
	\begin{align}\label{mc}
		&  s_n(X, T, d, \nu, \varepsilon) \nonumber\\
		&:= \min\bigg\{m\in\mathbb{N}: \nu\Big(\bigcup_{j=1}^{m}B_{\bar{d}_n}(x_j, \varepsilon)\Big)>1-\varepsilon \ \mbox{ for some $ x_1,\cdots, x_m\in X$}\bigg\},    
	\end{align}
	where $B_{\bar{d}_n}(x, \varepsilon)=\{y\in X:\bar{d}_n(x, y)<\varepsilon\}$. Note that $\nu(X)=1$. 
	We say that the measure complexity of $(X, T, \nu)$ is sub-polynomial if for any $\tau>0$, 
	\begin{equation}\label{sp}
		\liminf_{n\rightarrow\infty}\frac{s_n(X, T, d, \nu, \varepsilon)}{n^\tau}=0 \quad \mbox{ holds for any $\varepsilon>0$.}
	\end{equation}
	The following quoted criterion (although we do not use it here) provides a motivation for studying measure complexity.  Proposition~\ref{prop2.3} for isomorphic flows will be applied later.
	\begin{proposition}(\cite[Theorem 1.1]{HWY})
		If the measure complexity of $(X, T, \rho)$ is sub-polynomial for any $\rho\in M(X, T)$,
		then the M{\"o}bius disjointness conjecture holds for $(X, T)$.
	\end{proposition}
	Flow $(Y,{\rm T},\nu)$ is said to be measurably isomorphic to $(X,T,\rho)$ if there are Borel measurable subsets $X'\subset X$ and $Y'\subset Y$ satisfying $\rho(X')=\nu(Y')=1$, $T X' \subset X'$, ${\rm T}(Y') \subset Y'$
	such that $\phi\circ T = {\rm T} \circ \phi$ for some invertible measure-preserving map $\phi: X'\to Y'$.
	\begin{proposition}(\cite[Proposition 2.2]{HWY})\label{prop2.3}
		Suppose $(Y,{\rm T},\nu)$ is measurably isomorphic to $(X,T,\rho)$. If the measure complexity of $(Y,{\rm T}, \nu)$ is sub-polynomial, then so is $(X,T,\rho)$ and vice versa. 
	\end{proposition}
	
	\section{Some preparation}
	\subsection{$3$-dimensional Heisenberg nilmanifold}\label{ss3.1} The Heisenberg group $G$ is defined as in \eqref{GGamma}, consisting of all $3\times 3$ upper triangular real matrices with all diagonal elements equal to $1$. It is well-known that its commutator subgroup $G_1=[G,G]$ is equal to the center $Z(G)$ of $G$, and 
	\begin{eqnarray*}
		Z(G)=\begin{pmatrix}\begin{smallmatrix}
				1 & & \R\\ & 1 & \\ & & 1
			\end{smallmatrix}
		\end{pmatrix}.
	\end{eqnarray*} 
	Therefore,  any $g= \begin{pmatrix}\begin{smallmatrix} 1 & y & z \\ & 1 & x\\ & & 1\end{smallmatrix}\end{pmatrix}\in G$ may be decomposed into
	\begin{eqnarray*}
		g=\begin{pmatrix}\begin{smallmatrix} 1 & y &  \\ & 1 & x\\ & & 1\end{smallmatrix}\end{pmatrix}
		\begin{pmatrix}\begin{smallmatrix} 1 &  & z \\ & 1 & \\ & & 1\end{smallmatrix}\end{pmatrix},     
	\end{eqnarray*}
	and its inverse $g^{-1}$ can be expressed as
	\begin{eqnarray*}
		g^{-1} =\begin{pmatrix}\begin{smallmatrix} 1 &  & xy-z \\ & 1 & \\ & & 1\end{smallmatrix}\end{pmatrix}\begin{pmatrix}\begin{smallmatrix} 1 & -y &  \\ & 1 & -x\\ & & 1\end{smallmatrix}\end{pmatrix}
		=
		\begin{pmatrix}\begin{smallmatrix} 1 & -y &  \\ & 1 & -x\\ & & 1\end{smallmatrix}\end{pmatrix}\begin{pmatrix}\begin{smallmatrix} 1 &  & xy-z \\ & 1 & \\ & & 1\end{smallmatrix}\end{pmatrix}.  
	\end{eqnarray*}

	As in \cite{HLW}, We endow $G$ with a  metric $d_G$ which is left-invariant and satisfies 
	\begin{eqnarray*}
		d_G(g,g') \le \|\kappa(g^{-1}g')\|_\infty
		\qquad \mbox{ for any $g,g'\in G$,}  
	\end{eqnarray*}
	where $\kappa \begin{pmatrix}\begin{smallmatrix} 1 & y & z \\ & 1 & x\\ & & 1\end{smallmatrix}\end{pmatrix} = (x,y,z-xy)$ and $\|(x_1,x_2,x_3)\|_\infty = \max_{1\le i\le 3} |x_i|$ is the $\ell^\infty$-norm.

	Let $\Gamma \subset G$ be the lattice defined in \eqref{GGamma}. Then $\Gamma \backslash G$ is the $3$-dimensional Heisenberg nilmanifold. Sometimes we write $\bar{g}$ for $\Gamma g\in \Gamma \backslash G$. The metric $d_G$ descends to a metric $d_{\Gamma\backslash G}$ on $\Gamma \backslash G$, given by $d_{\Gamma\backslash G} (\bar{g}, \bar{g}') =\inf_{\gamma\in \Gamma} d_G(g,\gamma g')$, which is bi-invariant under $Z(G)$ (by the left-invariance of $ d_G$). 
	
	Let $\langle z \rangle := z-n_z$, where $n_z$ is the integer closest to $z$,  and  $\langle n+1/2 \rangle := 1/2$ for integer $n$.
	In particular, 
	\begin{eqnarray*}
		|\langle z\rangle | = \|z\|:= \min_{n\in\Z} |z-n|. 
	\end{eqnarray*}
	Unlike the abelian case (of $\mathbb{T}^2$), the coset $\overline{\begin{pmatrix}\begin{smallmatrix} 1 & b & c \\ & 1 & a \\ & & 1 \end{smallmatrix}\end{pmatrix}}$  may not have $\begin{pmatrix}\begin{smallmatrix} 1 & \langle b\rangle & \langle c\rangle \\ & 1 & \langle a\rangle \\ & & 1 \end{smallmatrix}\end{pmatrix}$ as a representative. Nonetheless, as  $\begin{pmatrix}\begin{smallmatrix}
			1 & & \Z\\ & 1 & \\ & & 1
	\end{smallmatrix} \end{pmatrix}\subset Z(G)\cap \Gamma$, for any $\bar{g}\in \Gamma\backslash G$, we have 
	\begin{equation}\label{eq3.1}
		\overline{g} \begin{pmatrix}\begin{smallmatrix} 1 & y & z \\ & 1 & x \\ & & 1 \end{smallmatrix}
		\end{pmatrix} = \overline{g} \begin{pmatrix}\begin{smallmatrix} 1 & y & \langle z\rangle \\ & 1 & x \\ & & 1
		\end{smallmatrix}\end{pmatrix}.     
	\end{equation}
	
	\subsection{A metric on $\mathbb{T}\times \Gamma \backslash G$}\label{ss3.2}
	
	Let $d_\T(t,t') := \|t-t'\|$ be the metric on $\T$ and define   the metric $d=d_{\mathbb{T}\times \Gamma \backslash G}$ on $\mathbb{T}\times \Gamma \backslash G$ by
	\begin{eqnarray}\label{eqd}
		d( (t,\overline{g}), (t',\overline{g'}))^2
		&:=&\|t-t'\|^2 +  d_{\Gamma\backslash G}  ( \overline{g}, \overline{g'})^2.
	\end{eqnarray}
	We derive some formulas to evaluate the metric $d$ and clearly it suffices to deal with $d_{\Gamma\backslash G}$. 
	\begin{enumerate}
		\item Suppose $g'= gY$ where $g, g', Y=\begin{pmatrix}\begin{smallmatrix} 1 & b & c \\ & 1 & a \\ & & 1 \end{smallmatrix}\end{pmatrix} \in G$. Let $Y' = \begin{pmatrix}\begin{smallmatrix} 1 & b & \langle c\rangle  \\ & 1 & a \\ & & 1 \end{smallmatrix}\end{pmatrix} \in G$. Then,
		\begin{eqnarray}\label{eq1}
			&& d_{\Gamma\backslash G}  ( \overline{g}, \overline{gY})
			= d_{\Gamma\backslash G}  ( \overline{g}, \overline{gY'}) 
			\le d_{G}  (g, gY')  \nonumber \\
			&\le&
			\left\| \kappa\begin{pmatrix}\begin{smallmatrix} 1 & b & \langle c\rangle  \\ & 1 & a \\ & & 1 \end{smallmatrix}\end{pmatrix}\right\|_\infty
			\le |a| + |b| + \|c\|.
		\end{eqnarray}
		\item More generally, for 
		$$
		g = \begin{pmatrix}\begin{smallmatrix} 1 & y & z  \\ & 1 & x \\ & & 1 \end{smallmatrix}\end{pmatrix},
		g^*= \begin{pmatrix}\begin{smallmatrix} 1 & y^* & z^*  \\ & 1 & x^* \\ & & 1 \end{smallmatrix}\end{pmatrix}, 
		Y= \begin{pmatrix}\begin{smallmatrix} 1 & b & c  \\ & 1 & a \\ & & 1 \end{smallmatrix}\end{pmatrix}, 
		Y^*= \begin{pmatrix}\begin{smallmatrix} 1 & b^* & c^*  \\ & 1 & a^* \\ & & 1 \end{smallmatrix}\end{pmatrix} \in G,
		$$
		we have for {\it any} $\mathfrak{z}, \mathfrak{z}^*\in \Gamma \cap Z(G)$,
		\begin{eqnarray}\label{eq2}
			&& d_{\Gamma\backslash G}  ( \overline{gY}, \overline{g^*Y^*}) 
			\le  d_{\Gamma\backslash G}  ( \overline{gY}, \overline{g^*Y}) + d_{\Gamma\backslash G}  ( \overline{g^*Y}, \overline{g^*Y^*})\nonumber\\ 
			&\le& d_{G}  (gY, g^*Y \mathfrak{z}) +  d_{G}  (g^*Y, g^*Y^*\mathfrak{z}^*)\nonumber\\
			&\le& \left\| \kappa\left(
			Y^{-1}g^{-1} g^* 
			Y \mathfrak{z}\right)\right\|_{\infty}
			+
			\left\| \kappa\left(
			Y^{-1}Y^* 
			\mathfrak{z}^*\right) \right\|_{\infty}.
		\end{eqnarray}
		A direct calculation yields
		$$
		g^{-1}g^* = \begin{pmatrix}\begin{smallmatrix} 1& y^*-y &  \\ & 1 & x^*-x\\ & & 1 \end{smallmatrix}\end{pmatrix}\begin{pmatrix}\begin{smallmatrix} 1&  & y(x-x^*) +z^*-z   \\ & 1 & \\ & & 1 \end{smallmatrix}\end{pmatrix},
		$$
		$$
		Y^{-1} =   \begin{pmatrix}\begin{smallmatrix} 1 & -b   &  \\ & 1 & -a  \\ & & 1 \end{smallmatrix}\end{pmatrix} \begin{pmatrix}\begin{smallmatrix} 1 &    & ab-c \\ & 1 &   \\ & & 1 \end{smallmatrix}\end{pmatrix}, 
		\quad 
		Y^* =   \begin{pmatrix}\begin{smallmatrix} 1 & b^*   &  \\ & 1 & a^*  \\ & & 1 \end{smallmatrix}\end{pmatrix} \begin{pmatrix}\begin{smallmatrix} 1 &    & c^* \\ & 1 &   \\ & & 1 
		\end{smallmatrix}\end{pmatrix}.
		$$
		If $w= ab+ y(x-x^*) +z^*-z$, then 
		\begin{eqnarray*}
			&& Y^{-1} g^{-1}g^* Y\\
			&=& 
			\begin{pmatrix}\begin{smallmatrix} 1 & -b   &  \\ & 1 & -a  \\ & & 1 \end{smallmatrix}\end{pmatrix}
			\begin{pmatrix}\begin{smallmatrix} 1& y^*-y &  \\ & 1 & x^*-x\\ & & 1 \end{smallmatrix}\end{pmatrix}
			\begin{pmatrix}\begin{smallmatrix} 1 & b   &  \\ & 1 & a  \\ & & 1 \end{smallmatrix}\end{pmatrix}
			\begin{pmatrix}\begin{smallmatrix} 1 &    & w  \\ & 1 &   \\ & & 1 \end{smallmatrix}\end{pmatrix}\\
			&=& \begin{pmatrix}\begin{smallmatrix} 1 & y^*-y   &  \\ & 1 & x^*-x  \\ & & 1 \end{smallmatrix}\end{pmatrix}
			\begin{pmatrix}\begin{smallmatrix} 1 &    & W \\ & 1 &   \\ & & 1 \end{smallmatrix}\end{pmatrix}
		\end{eqnarray*}
		with
		\begin{eqnarray*}
			W &=& w+ a (y^*-y) -b (x^*-x) - ab \\
			&=& (-y-b)(x^*-x) +a(y^*-y)+ (z^*-z).
		\end{eqnarray*}
		
		Therefore, for  suitable $\mathfrak{z}, \mathfrak{z}^*\in \Gamma \cap Z(G)$, we have
		\begin{eqnarray*}
			&& Y^{-1} g^{-1}g^* Y\mathfrak{z}
			= \begin{pmatrix}\begin{smallmatrix} 1 & y^*-y   &  \\ & 1 & x^*-x  \\ & & 1 \end{smallmatrix}\end{pmatrix}
			\begin{pmatrix}\begin{smallmatrix} 1 &    & W \\ & 1 &   \\ & & 1 \end{smallmatrix}\end{pmatrix} \mathfrak{z}
			= \begin{pmatrix}\begin{smallmatrix} 1 & y^*-y   &  \\ & 1 & x^*-x  \\ & & 1 \end{smallmatrix}\end{pmatrix}
			\begin{pmatrix}\begin{smallmatrix} 1 &    & \langle W\rangle  \\ & 1 &   \\ & & 1 \end{smallmatrix}\end{pmatrix} 
		\end{eqnarray*}
		and 
		\begin{eqnarray*}
			Y^{-1}Y^* \mathfrak{z}^* &=& \begin{pmatrix}\begin{smallmatrix} 1 & b^*-b   &  \\ & 1 & a^*-a  \\ & & 1 \end{smallmatrix}\end{pmatrix} \begin{pmatrix}\begin{smallmatrix} 1 &    & c^*-c - b(a^*-a) \\ & 1 &   \\ & & 1 
			\end{smallmatrix}\end{pmatrix}\mathfrak{z}^*\\
			&=& \begin{pmatrix}\begin{smallmatrix} 1 & b^*-b   &  \\ & 1 & a^*-a  \\ & & 1 \end{smallmatrix}\end{pmatrix} \begin{pmatrix}\begin{smallmatrix} 1 &    & \langle c^*-c - b(a^*-a)\rangle \\ & 1 &   \\ & & 1 
			\end{smallmatrix}\end{pmatrix}.
		\end{eqnarray*}
		Consequently, by \eqref{eq1} and \eqref{eq2}, 
		\begin{eqnarray}\label{eqdGammaG}
			&& d_{\Gamma\backslash G}  ( \overline{gY}, \overline{g^*Y^*}) \\
			&\le& |x^*-x|+ |y^*-y|  + (|y|+|b|)|x^*-x| + |a||y^*-y| + |z^*-z|  \nonumber\\
			& & + |b-b^*| + |a-a^*| + |b||a-a^*|+\|c^*-c\| \nonumber\\
			&=& (1+|y|+|b|)|x^*-x|+ (1+|a|)|y^*-y| + |z^*-z| \nonumber\\
			&& +(1+|b|)|a-a^*|+|b-b^*| +\|c^*-c\| .\nonumber
		\end{eqnarray}
	\end{enumerate}
	
	\subsection{Continued fractions and Convergents}\label{qk}
	
	Let $\alpha \in [0,1)$ and  $l_k/q_k$ be the $k$-th convergent in the continued fraction expansion of $\alpha$. In particular, $q_0=1$ and $q_1 = [1/\alpha]$. Assume $\alpha$ is irrational. The continued fraction expansion is infinite and  $\cQ=\{q_k\}_{k\ge 0}$ yields an infinite increasing sequence.  The following properties are known (cf. \cite[Lemma 4.1]{HLW} and \cite[Section 3 and (13)]{Faveri}).  
	\begin{itemize}[leftmargin=4mm]
		\item If $|\alpha - l/q|< 1/(2q^2)$ where $l\in \Z$ and $q\in \Z\setminus\{0\}$, then  for some $k\ge 1$,
		\begin{align}\label{p1}
			l/q = l_k/q_k.
		\end{align}
		\item Let $k\ge 1$ be any integer. Then
		\begin{align}\label{p2}
			\frac1{2q_{k+1}} < \|q_k \alpha\| < \frac1{q_{k+1}}, \ \mbox{ and  } \ 
			\|q_k \alpha \|\le \|q\alpha\|, \ \forall \ 1\le q< q_{k+1}.
		\end{align}
		\item For $k\ge 0$ and $1\le c\le q_k$, 
		\begin{align}\label{p3}
			\sum_{q_k\le |q|< q_{k+1}} \frac1{q^2} \min \bigg(\frac1{\|q\alpha\|^2}, c^2\bigg) \ll \frac{c}{q_k}.
		\end{align}
		\item For $k\ge 1$, 
		\begin{align}\label{p4}
			\sum_{0<|q|< q_k} \|q\alpha\|^{-1} \asymp q_k\log (q_k+1) \ \mbox{ and } \ 
			\sum_{0<|q|< q_{k+1}} \|q\alpha\|^{-2} \asymp q_{k+1}^2.
		\end{align}
	\end{itemize}
	The implied constants in $\ll$ and $\asymp$ of \eqref{p3} and \eqref{p4} are absolute.

	
	\section{Initial steps for the proof of Theorem~\ref{intd}}
	
	We consider the family $\cF$ of  Lipschitz continuous functions $g$ on ${\rm X}:= \T\times \Gamma \backslash G$. For any $\nu \in M({\rm X}, S)$ with $S=S_\alpha$, $n\in \N$ and $g\in \cF$, we have
	\begin{eqnarray*}
		\|g\circ S^n-g\|_{L^2(\nu)}^2 \ll_g \int_{\T\times \Gamma\backslash G} d\big( (t,\overline{g}), S^n((t,\overline{g})) \big)^2 \textup{d} \nu(t,\overline{g}).
	\end{eqnarray*}
	Suppose there exist $\varepsilon>\delta >0$ and an unbounded sequence $\{r_n\}_{n\ge 1}\subset \N$ for which 
	\begin{eqnarray}\label{A1}
		\int_{\T\times \Gamma\backslash G} d\big( (t,\overline{g}), S^{jr_n}((t,\overline{g})) \big)^2 \textup{d} \nu(t,\overline{g}) \ll_\varepsilon r_n^{-\varepsilon}, \qquad \forall \ |j|\le r_n^\delta.
	\end{eqnarray}
	Clearly \eqref{A1} implies \eqref{PRr} and the PR rigidity follows.
	
	We shall apply \S\S~\ref{ss3.2} to evaluate the left-hand side of \eqref{A1}, and beforehand, we need explicit formulas for $S^n$.
	
	\subsection{Formula for the iteration of $S$}\label{ss4.1}
	Let $l_\alpha$ be the translation $l_\alpha(t)=t+\alpha$. From \S\S~\ref{ss3.1}, we express \eqref{S} as $S(t,\overline{g}) = (l_\alpha (t), \overline{g} h_0(t) w_0(t))$ where 
	$$
	h_0(t) =  \begin{pmatrix}\begin{smallmatrix} 1 & \varphi(t)   &  \\ & 1 & \eta(t)  \\ & & 1\end{smallmatrix}\end{pmatrix}
	\quad
	\mbox{ and }
	\quad 
	w_0(t)= \begin{pmatrix}\begin{smallmatrix} 1 & \phantom{\varphi(t)}  & \psi(t) \\ & 1 & \phantom{\eta(t)} \\ & & 1\end{smallmatrix}\end{pmatrix}.
	$$
	The $n$-fold composition $S^n$ is clearly given by
	\begin{eqnarray*}
		S^n(t,\overline{g})
		&=&
		\big(l_\alpha^n (t), \overline{g}(h_0h_1\cdots h_{n-1})\cdot (w_0w_1\cdots w_{n-1})(t)\big),
	\end{eqnarray*}
	where $h_r(t):= h_0\circ l_\alpha^r(t) = h_0(t+r\alpha)$ and $w_r := w_0\circ l_\alpha^r$.
	Define for $n\ge 1$,
	\begin{eqnarray}\label{pxoh}
		&& \left\{
		\begin{array}{rcl}
			\Phi_n(t)&=& \ds \sum_{r=0}^{n-1}\varphi\circ l_\alpha^r(t) = \sum_{r=0}^{n-1}\varphi(t+r\alpha),\vspace{2mm}\\
			\xi_n(t) &=& \ds \sum_{r=0}^{n-1} \eta\circ l_\alpha^r (t) = \sum_{r=0}^{n-1}\eta(t+r\alpha),\vspace{2mm}\\
			\Psi_n(t)&=& \ds\Omega_n(t)  + H_n(t)
			= \sum_{r=0}^{n-1}\psi(t+r\alpha) + \sum_{j=1}^{n-1} \sum_{r=0}^{j-1} \varphi(t+r\alpha) \cdot\eta(t+j\alpha).
		\end{array}
		\right.
	\end{eqnarray}
	Tacitly empty summation means 0, which happens when $n=1$, and $\Omega_n(t)$ and $H_n(t)$ are respectively defined by the sum over $r$ and the double sum over $j,r$. We 
	observe that the term $(w_0w_1\cdots w_{n-1})(t)$ is an central element whose upper rightmost element is $\Omega_n(t)$ and
	\begin{eqnarray*}
		(h_0h_1\cdots h_{n-1})(t) = \begin{pmatrix}\begin{smallmatrix}
				1 & \Phi_n(t)  &  \\ & 1 & \xi_n(t) \\ & & 1
		\end{smallmatrix} \end{pmatrix} \begin{pmatrix}\begin{smallmatrix} 1 &   & H_n(t) \\ & 1 &  \\ & & 1\end{smallmatrix}\end{pmatrix},  
	\end{eqnarray*}
	hence deduce that 
	\begin{eqnarray}\label{Sn}
		S^n(t,\overline{g})
		=
		\left(l_\alpha^n (t), \overline{g} \begin{pmatrix}\begin{smallmatrix} 1 & \Phi_n(t)   & \Psi_n(t) \\ & 1 & \xi_n(t)  \\ & & 1\end{smallmatrix} \end{pmatrix}\right).
	\end{eqnarray}
	
	\subsection{Bounding the metric on $\mathbb{T}\times \Gamma \backslash G$}
	Using \eqref{eqd}, \eqref{Sn}, \eqref{eq3.1} and \eqref{eqdGammaG},  we have 
	\begin{eqnarray}\label{d}
		d\big( (t,\overline{g}), S^n((t,\overline{g})) \big)^2
		&=&\|n\alpha\|^2 +  d_{\Gamma\backslash G}      \left( \overline{g}, \overline{g}\begin{pmatrix}\begin{smallmatrix}
				1 & \Phi_n(t)   & \langle \Psi_n(t) \rangle \\ & 1 & \xi_n(t)  \\ & & 1 \end{smallmatrix}\end{pmatrix}  \right)^2 \nonumber\\
		&\ll& \|n\alpha\|^2+    |\Phi_n(t)|^2  + |\xi_n(t)|^2 + \|\Omega_n(t)\|^2  + \|H_n(t)\|^2
	\end{eqnarray}
	as $|\langle \Psi_n(t)\rangle | = \|\Psi_n(t)\|$ and $\|\cdot \|$ is a norm. By \eqref{d}, we obtain
	\begin{eqnarray*}
		&&\int_{\T\times \Gamma\backslash G} d\big( (t,\overline{g}), S^n((t,\overline{g})) \big)^2 \textup{d} \nu(t,\overline{g})\nonumber\\
		&\ll& \|n\alpha\|^2+  \int_{\T\times \Gamma\backslash G}  \big(|\Phi_n(t)|^2  + | \xi_n(t)|^2 + \|\Omega_n(t)\|^2  +\|H_n(t)\|^2 \big)  \textup{d} {\nu}(t,\overline{g}).
	\end{eqnarray*}
	The pushforward measure $\pi_*\nu$ of the projection map $\pi(t,\overline{g})=t$ is a Borel probability measure on the torus $\T$. When $\alpha$ is irrational, this measure is rotational invariant and hence equal to the normalized Haar measure (which is also the Lebesgue measure) on $\T$.   
	Since the integrand depends only on $t$ (and not on $\overline{g}$), we obtain that
	\begin{eqnarray}\label{target1}
		&& \int_{\T\times \Gamma\backslash G} d\big( (t,\overline{g}), S^n((t,\overline{g})) \big)^2 \textup{d} \nu(t,\overline{g})
		\le  \|n\alpha\|^2+ I_\Phi(n) + I_\xi(n) + I_\Omega(n) + I_H(n),
	\end{eqnarray}
	where 
	\begin{eqnarray*}
		\ds I_\Phi(n) = \int_\T |\Phi_n(t)|^2 \textup{d}t, &&  
		\ds I_\xi(n) = \int_\T |\xi_n(t)|^2 \textup{d}t,\\
		\ds I_\Omega(n) = \int_\T \|\Omega_n(t)\|^2 \textup{d}t, && 
		\ds I_H(n) = \int_\T \|H_n(t)\|^2 \textup{d}t.
	\end{eqnarray*}
	
	In view of \eqref{pxoh} and the work in \cite{HW}, we shall only deal with $I_H(n)$ in detail.  The outline of the treatment for $I_\Phi(n), I_\xi(n)$ and $I_\Omega (n)$ are the same as in \cite{Faveri}. 
	
	\subsection{Treatments for $H_n(t)$}
	We decompose $H_n(t)$ with the Fourier expansion  
	\begin{eqnarray*}
		\varphi(t) =\sum_{u\neq0} \widehat{\varphi}(u) e(ut),   \qquad   \quad
		\eta(t) =\sum_{v\neq0} \widehat{\eta}(v) e(vt)
	\end{eqnarray*} 
	of $\varphi$ and $\eta$.   Plainly, we have
	\begin{eqnarray}\label{H}
		H_n(t)
		&=&\sum_{j=1}^{n-1} \sum_{r=0}^{j-1} \varphi(t+r\alpha)  \eta(t+j\alpha)\nonumber\\
		&=&\sum_{0\neq u,v\in \Z} \widehat{\varphi}(u)\widehat{\eta}(v) e((u+v)t) \sum_{j=1}^{n-1} \sum_{r=0}^{j-1}  e((ur+vj)\alpha) \nonumber\\
		&=& \sum_{u+v=0} + \sum_{u+v\neq 0} =:\Sigma_{n,0} + \Sigma_n(t), \mbox{ say}.  
	\end{eqnarray}
	Apparently the double sum over $j$ and $r$ in $\Sigma_{n,0}$ is 
	\begin{eqnarray}\label{w0}
		\sum_{j=1}^{n-1} \sum_{r=0}^{j-1}  e(u(r-j)\alpha)
		&=& 
		\sum_{j=1}^{n-1} e(-uj\alpha) \frac{1-e(uj\alpha)}{1-e(u\alpha)} \nonumber\\
		&=& \frac1{1-e(u\alpha)}\bigg(\sum_{j=0}^{n-1} e(-uj\alpha)-\sum_{j=0}^{n-1} 1\bigg)
		\nonumber\\
		&=&\frac1{1-e(u\alpha)} \Big\{\frac{1-e(-nu\alpha)}{1-e(-u\alpha)} - n\Big\},
	\end{eqnarray}
	and the double sums in $\Sigma_n(t)$ can be expressed in two ways:
	\begin{eqnarray}\label{w1}
		\sum_{j=1}^{n-1} \sum_{r=0}^{j-1}  e((ur+vj)\alpha) 
		&=& \sum_{j=1}^{n-1}  e(vj\alpha)\frac{1-e(uj\alpha)}{1-e(u\alpha)}\nonumber\\
		&=& \frac1{1-e(u\alpha)}\bigg(\sum_{j=0}^{n-1} e(vj\alpha)-\sum_{j=0}^{n-1} e((u+v)j\alpha)\bigg)
		\nonumber\\
		&=&\frac1{1-e(u\alpha)} \Big\{ \frac{1-e(nv\alpha)}{1-e(v\alpha)} - \frac{1-e(n(u+v)\alpha)}{1-e((u+v)\alpha)}\Big\}
	\end{eqnarray}
	or
	\begin{eqnarray}\label{w2}
		&& \sum_{j=1}^{n-1} \sum_{r=0}^{j-1}  e((ur+vj)\alpha) =\sum_{r=0}^{n-2} \sum_{j=r+1}^{n-1}e((ur+vj)\alpha) \nonumber\\
		&=& \sum_{r=0}^{n-2}  e(ur\alpha)e(v(r+1)\alpha)\frac{1-e(v(n-r-1)\alpha)}{1-e(v\alpha)}\nonumber\\
		&=& \frac1{1-e(v\alpha)}\bigg(e(v\alpha)\sum_{r=0}^{n-2}  e((u+v)r\alpha)-e(nv\alpha)\sum_{r=0}^{n-2}  e(ur\alpha)\bigg)\nonumber\\     
		&=& \frac{1}{1-e(v\alpha)} \Big\{ \frac{1-e(n(u+v)\alpha)}{1-e((u+v)\alpha)}e(v\alpha) - \frac{1-e(nu\alpha)}{1-e(u\alpha)} e(nv\alpha)\Big\}.
	\end{eqnarray}
	Here we have inserted $e(v\alpha) e((u+v)(n-1)\alpha)-e(nv\alpha)e(u(n-1)\alpha)$, which equals zero, in the bracket of the second last line in order to get the formula in \eqref{w2}.

	\section{Proof of Theorem~\ref{intd}}
	We need to construct an unbounded sequence of natural numbers $\{r_n\}$ for which \eqref{A1} holds.  
	Let $\cQ:=\{q_m\}_{m\ge 0}$ be the set of the denominators of the convergents arising from the continued fraction expansion of $\alpha$  as in \S\S~\ref{qk}.   The $C^{1+\varepsilon}$ and $C^{2+2\varepsilon}$ smoothness imply that given any  $\varepsilon\in (0,\frac1{100})$,  
	\begin{eqnarray}\label{eqB}
		\widehat{\varphi}(u), \widehat{\eta}(u) \ll |u|^{-B} \quad  \mbox{ and }\quad \widehat{\psi}(u)\ll |u|^{-B/2}
	\end{eqnarray} 
	for $u\neq 0$, where $B=2+ 2\varepsilon>2$. Set  $\theta = \frac18\varepsilon$, so  
	\begin{align}\label{eqBt}
		B=2+16\theta \quad \mbox{ and } \quad  0<\theta <\frac1{12}(B-2)=\frac16\varepsilon.
	\end{align}
	We shall construct an unbounded subset $\cQ''\subset \cQ$ and  an integer $v_m \in [1, q_m^{\theta/2}]$  for every $q_m\in \cQ''$ such that the right-hand side of \eqref{target1} is  
	$\ll  r_m^{-\varepsilon/120}$ when $n=jr_m$,  where $r_m:=v_mq_m$ and $|j|< r_m^{\theta/10}$. The construction depends on the separation of consecutive $q_m$'s in $\cQ$ and hence on $\alpha$. 
	
	Let     $\cQ':= \{ q_m\in \cQ: q_{m+1} \ge  q_m^{2+2\theta}\}$. 
	If $\cQ'$ is finite, then  $q_m^{2+2\theta} > q_{m+1}$ for all $m\ge  m_0$, for some $m_0$. 
	Define
	\begin{eqnarray*}
		\cQ'' = \left\{
		\begin{array}{ll} 
			\cQ' & \mbox{ if $\cQ'$ is infinite},\\
			\{q_l\in \cQ: l\ge m_0\} & \mbox{ if $\cQ'$ is finite}.
		\end{array}
		\right.
	\end{eqnarray*}
	If we set $s_m:= jv_m$ where $|j|< r_m^{\theta/10}$, then $s_m \le q_m^\theta$ and $jr_m = s_mq_m$. Therefore, we shall estimate the integrals in \eqref{target1} for $n \in \{s_mq_m: s_m\in [1,q_m^\theta], q_m \in \cQ''\} $. As discussed after \eqref{target1}, 
	we invoke from \cite{HW} for the estimates of $I_\Phi, I_\xi, I_\Omega$. The integral
	$I_H(n)$ is evaluated using the following lemma which will be proved  with tools from  \cite{KLR}, \cite{Faveri} and \cite{HW} in the next section.
	
	\begin{lemma}\label{lem31} Let $q_m\in \cQ''$ and set
		$$\lambda(q_m) := \sum_{0< |u|< q_m} \frac{\widehat{\varphi}(u)\widehat{\eta}(-u)}{1-e(-u\alpha)}.
		$$
		For $n= s_mq_m$, where $s_m\in [1,q_m^\theta]$ is any integer, we have
		$$
		\int_{\mathbb{T}} \|H_n(t)\|^2\,\textup{d} t \ll \|s_mq_m \lambda(q_m) \|^2 + q_m^{-\theta}.
		$$
	\end{lemma}
	
	Applying the argument for $\int_{\mathbb T} |\Phi(n,t)|^2\,\textup{d} t$
	in \cite[p.479-480]{HW}, we can derive that for $q_m\in \cQ''$ and  $n=s_m q_m$, 
	\begin{eqnarray*}
		\int_{\T} |\Phi_n(t)|^2\, \textup{d} t 
		+ \int_{\T} |\xi_n(t)|^2\,  \textup{d} t
		\ll q_m^{-\varepsilon/2},
	\end{eqnarray*} 
	and together with a separation of the $0$-th Fourier coefficient,
	\begin{eqnarray*}
		\int_{\T}   \|\Omega_n(t)\|^2 \,  \textup{d} t
		&\ll& \|s_mq_m \widehat{\psi}(0)\|^2 + 
		\int_{\T}   |\Omega_n(t) - n\widehat{\psi}(0)|^2 \,  \textup{d} t\\
		&\ll&  \|s_mq_m \widehat{\psi}(0)\|^2 + q_m^{-\varepsilon/2}.   
	\end{eqnarray*}
	
	By Dirichlet's approximation theorem, we can choose $1\le v_m \le q_m^{\theta/2}$ such that
	$$ \|v_mq_m\alpha\|^2, \|v_mq_m \widehat{\psi}(0)\|^2, \|v_mq_m \lambda(q_m) \|^2 \le q_{m}^{-\theta/3}. 
	$$
	Thus, $ \|s_mq_m\alpha\|^2$, $\|s_mq_m \widehat{\psi}(0)\|^2$, $\|s_mq_m \lambda(q_m) \|^2$ are $ \le q_{m}^{-\theta/12}$ for $s_m=jv_m$ with $|j|<r_m^{\theta/10}$ ($\le q_m^{\theta/8}$  as $r_m = v_mq_m \le q_m^{5/4}$).
	Consequently, the right-side of \eqref{target1} is $\ll q_m^{-\varepsilon/2} + q_m^{-\theta/12} \le r_m^{-\varepsilon/120}$, recalling $\theta=\frac18\varepsilon$. This leads to \eqref{A1} by taking $\delta = \frac1{ 10}\theta$ and replacing $\varepsilon$ with $120\varepsilon$.  Theorem~\ref{intd} follows readily.

	\section{Proof of Lemma~\ref{lem31}}
	By \eqref{H} and \eqref{w0}, we get
	\begin{equation}\label{intH}
		\int_{\mathbb{T}} \|H_n(t)\|^2\,\textup{d} t \ll \|\Sigma_{n,0}\|^2 + \int_{\mathbb{T}} |\Sigma_n(t)|^2\,\textup{d} t,    
	\end{equation}
	where 
	\begin{eqnarray}
		\Sigma_{n,0} &=& \sum_{0\neq u\in \Z} \frac{\widehat{\varphi}(u)\widehat{\eta}(-u)}{1-e(u\alpha)}  \Big\{\frac{1-e(-nu\alpha)}{1-e(-u\alpha)} - n\Big\}, \nonumber\\
		\Sigma_n(t)  &=& \sum_{\substack{0\neq u,v\in \Z\\ u+v\neq 0}} \widehat{\varphi}(u)\widehat{\eta}(v) e((u+v)t)   \sum_{j=1}^{n-1} \sum_{r=0}^{j-1}  e((ur+vj)\alpha). \label{snt}
	\end{eqnarray}
	Recall $n=s_mq_m$. Our task will be showing 
	\begin{eqnarray}\label{lemt}
		\Sigma_{n,0} = -n  \Sigma_{q_m,0}^{<,2} +O(q_m^{-\theta})\quad \mbox{ and } \quad \int_{\T}     |\Sigma_n(t)|^2\,\textup{d} t \ll q_m^{-\theta},
	\end{eqnarray}
	where $\Sigma_{q_m,0}^{<,2}$, defined in \eqref{sl2} below, is a real number independent of $s_m$. 
	
	We are going to separate into two cases to handle \eqref{intH} according as $\cQ''=\cQ'$ or not, corresponding to (i) $\# \cQ'=\infty$ or (ii) $\# \cQ'<\infty$. \S\S~\ref{1a}-\ref{1b} and \S\S~\ref{1c}-\ref{1d} below are devoted to  Cases (i) and (ii) respectively. The (latter) case for finite $\cQ'$ requires more delicate analysis. Note from \eqref{eqBt} that the exponent $B$ satisfies $B=2+2\varepsilon >2+12 \theta$, hence $\sum_{q_m\in S} q_m^{1+\Delta -B}<\infty$ for any $\Delta <12 \theta$ and any (infinite) subset $S\subset \N$.
	
	{\bf Case (i):} $\# \cQ'=\infty$. In this case, $n=s_mq_m$,  where $q_m\in \cQ'$ and $s_m\in [1, q_m^\theta]$ and thus $q_{m+1} \ge q_m^{2+2\theta}$. 
	
	\subsection{Case (i): Evaluation of $\Sigma_{n,0}$}\label{1a} We split 
	\begin{eqnarray}\label{sn0}
		\Sigma_{n,0} &=& \sum_{0\neq u\in \Z} \frac{\widehat{\varphi}(u)\widehat{\eta}(-u)}{1-e(u\alpha)}  \Big\{\frac{1-e(-nu\alpha)}{1-e(-u\alpha)} - n\Big\} \nonumber\\
		&=& \sum_{0<|u|<q_m} + \sum_{|u|\ge q_m} =:\Sigma_{n,0}^{<} +\Sigma_{n,0}^{\ge}, \mbox{ say}.
	\end{eqnarray}
	Clearly, bounding  trivially with \eqref{w0}, \eqref{eqB} and \eqref{eqBt} yields 
	\begin{eqnarray}\label{0b}
		\Sigma_{n,0}^{\ge} &\ll&  n^2 \sum_{u\ge q_m} u^{-2B} \ll q_m^{3+2\theta - 2B}\ll q_m^{-\theta}.
	\end{eqnarray}
	
	Write 
	\begin{eqnarray*}
		\Sigma_{n,0}^{<} = \Sigma_{n,0}^{<,1} - n\Sigma_{q_m,0}^{<,2},
	\end{eqnarray*}
	where
	\begin{eqnarray*}
		\Sigma_{n,0}^{<,1} := \sum_{0< |u|< q_m} \frac{\widehat{\varphi}(u)\widehat{\eta}(-u)}{1-e(u\alpha)}  \frac{1-e(-nu\alpha)}{1-e(-u\alpha)}     
	\end{eqnarray*}
	and 
	\begin{eqnarray}\label{sl2}
		\Sigma_{q_m,0}^{<,2}:= \sum_{0< |u|< q_m} \frac{\widehat{\varphi}(u)\widehat{\eta}(-u)}{1-e(u\alpha)} .
	\end{eqnarray}
	Here we note that $\Sigma_{q_m,0}^{<,2}$ is independent of $s_m$. 
	
	Using $1-e(nu\alpha)\ll \|us_m q_m\alpha\|<|us_m|\|q_m\alpha\|\le |u|q_m^\theta q_{m+1}^{-1}$ and  $\|u\alpha\| > \|q_{m-1} \alpha\|\gg q_{m}^{-1}$ for $0<|u|< q_{m}$ by \eqref{p2}, we infer with \eqref{eqB} that 
	\begin{eqnarray*}
		\Sigma_{n,0}^{<,1}\ll \frac{q_m^\theta}{q_{m+1}}\sum_{0< u< q_m} u^{1-2B}\|u\alpha\|^{-2}
		\ll \frac{q_m^{2+\theta}}{q_{m+1}}\sum_{0< u< q_m} u^{1-2B} \ll \frac{q_m^{2+\theta}}{ q_{m+1}}. 
	\end{eqnarray*}
	Hence, we obtain by $q_{m+1} \ge q_m^{2+2\theta}$ that
	\begin{eqnarray*}
		\Sigma_{n,0} = -n  \Sigma_{q_m,0}^{<,2} +O(q_m^{-\theta} + q_m^{2+\theta}q_{m+1}^{-1})= -n  \Sigma_{q_m,0}^{<,2} +O(q_m^{-\theta}) .
	\end{eqnarray*}

	\subsection{Case (i): Evaluation of $\Sigma_n(t)$}\label{1b} We split $\Sigma_n(t)$ into four subsums according to $|u|\le q_m/4$ or $|v|\le  q_m/4$ or not,
	\begin{eqnarray*}
		\Sigma_n(t) := \Sigma_n^{>,>}(t) + \Sigma_n^{\le ,>}(t) + \Sigma_n^{>,\le }(t) + \Sigma_n^{\le ,\le }(t), 
	\end{eqnarray*}
	which leads to 
	\begin{eqnarray*}
		\int_{\mathbb{T}}|\Sigma_n(t) |^2\,\textup{d} t\ll \sum_{a,b\in \{>,\leq \}} I^{a,b},
	\end{eqnarray*}
	where $I^{a,b}:=\int_{\mathbb{T}}| \Sigma_n^{a,b}(t)|^2\,\textup{d} t$.
	
	\subsubsection{Both $|u|$ and $|v|> q_m/4$} In \eqref{snt}, we bound the double sum over $j,r$ trivially (in the same fashion as in \eqref{0b}) and get 
	\begin{eqnarray*}
		I^{>,>} 
		&\le&  n^4\sum_{\substack{|u_1|, |u_2|,|v_1|,|v_2| > q_m/4 \\ u_1+v_1= u_2+v_2}} |\widehat{\varphi}(u_1)\widehat{\eta}(v_1)\widehat{\varphi}(u_2)\widehat{\eta}(v_2)|\\  
		&\ll&  q_m^{4+4\theta}\sum_{\substack{|u_1|, |u_2|,|v_1|,|v_2| > q_m/4 \\ u_1+v_1= u_2+v_2}} |u_1u_2v_1v_2|^{-B}\\  
		&\ll&  q_m^{4+4\theta-B}\sum_{|u_1|, |u_2|,|v_1| > q_m/4} |u_1u_2v_1|^{-B}\ll q_m^{7+4\theta-4B} \ll q_m^{-1-4\theta}.
	\end{eqnarray*}
	
	\subsubsection{$0<|u|\le q_m/4$ and $|v|> q_m/4$} Let 
	\begin{eqnarray*}
		f(u,v) := \frac{1}{1-e(u\alpha)} \Big\{ \frac{1-e(nv\alpha)}{1-e(v\alpha)} - \frac{1-e(n(u+v)\alpha)}{1-e((u+v)\alpha)} \Big\}.
	\end{eqnarray*}
	By \eqref{w1} and \eqref{snt}, we get
	\begin{eqnarray*}
		I^{\le ,>} 
		&=&  \sum_{\substack{0<|u_1|, |u_2|\le q_m/4\\ |v_1|,|v_2| > q_m/4 \\ u_1+v_1= u_2+v_2}} \widehat{\varphi}(u_1)\widehat{\eta}(v_1)f(u_1,v_1)\overline{\widehat{\varphi}(u_2)\widehat{\eta}(v_2) f(u_2,v_2)}. 
	\end{eqnarray*}
	As $|ab|\le |a|^2+|b|^2$ and $|a+b|^2\le 2|a|^2+2|b|^2$, we deduce that
	\begin{eqnarray}\label{I1}
		I^{\le ,>} \ll S_1 + S_2,
	\end{eqnarray}
	where 
	\begin{eqnarray}\label{S12}
		\begin{array}{rcl}
			S_1 &:=& \displaystyle \sum_{\substack{0<|u_1|, |u_2|\le q_m/4\\ |v_1|,|v_2| > q_m/4 \\ u_1+v_1= u_2+v_2}} \frac{|u_1v_1u_2v_2|^{-B}}{|1-e(u_1\alpha)||1-e(u_2\alpha)|}  \Big|\frac{1-e(n(u_1+v_1)\alpha)}{1-e((u_1+v_1)\alpha)}\Big|^2 , \\ 
			S_2 &:=& \displaystyle \sum_{\substack{0<|u_1|, |u_2|\le q_m/4\\ |v_1|,|v_2| > q_m/4 \\ u_1+v_1= u_2+v_2}} \frac{|u_1v_1u_2v_2|^{-B}}{|1-e(u_1\alpha)||1-e(u_2\alpha)|} \Big|\frac{1-e(nv_1\alpha)}{1-e(v_1\alpha)} \Big|^2.
		\end{array}
	\end{eqnarray}
	Write $l=|u_1+v_1|$ and note that  for $0<|u|<q_m$,  $\|u\alpha\|>\|q_{m-1}\alpha\|\gg q_m^{-1}$ by \eqref{p2}, and 
	$$
	\frac{1-e(nl\alpha)}{1-e(l\alpha)} \ll \min\Big(\frac{s_m l q_{m+1}^{-1}}{\|l\alpha\|},n\Big)\le \min \Big(\frac{q_m^\theta l}{q_{m+1}\|l\alpha\|}, n\Big).
	$$ 
	We split the sum into two pieces according to $l< q_{m+1}$ or not, and (use \eqref{p2} again to) bound $|1-e(u_1\alpha)|^{-1}|1-e(u_2\alpha)|^{-1}$ by $q_m^2$, thus 
	\begin{eqnarray*}
		S_1 
		&\ll& \frac{q_m^{2+2\theta}}{q_{m+1}^2} \sum_{1\le l<q_{m+1}}\frac1{\|l\alpha\|^2}  \sum_{\substack{0<|u_1|, |u_2|\le q_m/4\\ |v_1|,|v_2| > q_m/4 \\ |u_1+v_1|= l=|u_2+v_2|}} |u_1v_1u_2v_2|^{-B}|u_1+v_1|^2\\
		&& + n^2 q_m^2 \sum_{l\ge q_{m+1}} \sum_{\substack{0<|u_1|, |u_2|\le q_m/4\\ |v_1|,|v_2| > q_m/4 \\ |u_1+v_1|= l=|u_2+v_2|}} |u_1v_1u_2v_2|^{-B}\\
		&=:& S_{1,1}+S_{1,2}, \mbox{ say}.
	\end{eqnarray*}
	For fixed $l$ and $u_1$, there are two solutions of $v_1$, and $v_2$ is governed in the same fashion. Note that $|v_1|^{-B}|v_2|^{-B}|u_1+v_1|^2\ll |v_1|^{2-B} |v_2|^{-B}\ll q_m^{2-B}q_m^{-B}$ (see \eqref{eqBt} for $B$).  Thus, 
	\begin{eqnarray*}
		S_{1,1} \ll \frac{q_m^{4+2\theta-2B}}{q_{m+1}^2} \sum_{1\le l<q_{m+1}} \frac1{\|l\alpha\|^2}  \sum_{0<|u_1|, |u_2|\le q_m/4} |u_1u_2|^{-B}\ll q_m^{4+2\theta-2B},
	\end{eqnarray*}
	by observing \eqref{p4} and the bound $O(1)$ for the double sum over $u_1,u_2$. Similarly, with $n^2q_m^2\le q_m^{4+2\theta}$, we have
	\begin{eqnarray*}
		S_{1,2} \ll q_m^{4+2\theta} \sum_{l\ge q_{m+1}} l^{-2B}  \sum_{0<|u_1|, |u_2|\le q_m/4} |u_1u_2|^{-B}\ll q_m^{4+2\theta} q_{m+1}^{1-2B}.
	\end{eqnarray*}
	Following the same line of treatment for $S_{1,1}$ and $S_{1,2}$, we obtain that
	\begin{eqnarray*}
		S_2   &\ll& \frac{q_m^{2+2\theta}}{q_{m+1}^2} \sum_{q_m/4< |v_1|<q_{m+1}}\frac1{\|v_1\alpha\|^2}  \sum_{\substack{0<|u_1|, |u_2|\le q_m/4\\ |v_2| > q_m/4 \\ u_1+v_1=u_2+v_2}} |u_1v_1u_2v_2|^{-B}|v_1|^2\\
		&& + n^2 q_m^2 \sum_{|v_1|\ge q_{m+1}} \sum_{\substack{0<|u_1|, |u_2|\le q_m/4\\ |v_2| > q_m/4 \\ u_1+v_1=u_2+v_2}} |u_1v_1u_2v_2|^{-B} .
	\end{eqnarray*}
	When $v_1,u_1,u_2$ are fixed, the value of $v_2$ is determined. Thus, the inner sum 
	$$\sum_{\substack{0<|u_1|, |u_2|\le q_m/4\\ |v_2| > q_m/4 \\ u_1+v_1=u_2+v_2}} |u_1v_1u_2v_2|^{-B}|v_1|^2\ll |v_1|^{2-B} q_m^{-B} \sum_{0<|u_1|, |u_2|\le  q_m/4} |u_1u_2|^{-B}\ll q_m^{2-2B}.
	$$
	Therefore, 
	\begin{eqnarray*}
		S_2     
		&\ll &  q_m^{4+2\theta-2B} + q_m^{4+2\theta} q_{m+1}^{1-2B}.
	\end{eqnarray*}
	Altogether we get, by using $q_m^{2+2\theta} \le  q_{m+1}$,
	\begin{eqnarray*}
		I^{\le ,>} \ll q_m^{4+2\theta-2B} + q_m^{4+2\theta} q_{m+1}^{1-2B}\ll q_m^{-2\theta}.
	\end{eqnarray*}
	
	\subsubsection{$|u|> q_m/4$ and $|v|\le q_m/4$} Let 
	\begin{eqnarray}\label{g}
		g(u,v) := \frac1{1-e(v\alpha)} \Big\{\frac{1-e(n(u+v)\alpha)}{1-e((u+v)\alpha)} e(v\alpha) -\frac{1-e(nu\alpha)}{1-e(u\alpha)} e(nv\alpha)  \Big\}.
	\end{eqnarray}
	With \eqref{snt} and \eqref{w2}, we have
	\begin{eqnarray*}
		I^{>,\le} 
		&=&  \sum_{\substack{|u_1|, |u_2|> q_m/4\\ 0<|v_1|,|v_2| \le  q_m/4 \\ u_1+v_1= u_2+v_2}} \widehat{\varphi}(u_1)\widehat{\eta}(v_1)g(u_1,v_1)\overline{\widehat{\varphi}(u_2)\widehat{\eta}(v_2) g(u_2,v_2)}
	\end{eqnarray*}
	and hence deduce by relabelling $u_1$ by $v_1$ and $v_1$ by $u_1$ that
	\begin{eqnarray*}
		I^{>,\le} \ll S_1 +S_2, 
	\end{eqnarray*}
	where $S_1$ and $S_2$ are defined as in \eqref{S12}. 
	Consequently with $q_m^{2+2\theta} \le  q_{m+1}$,
	\begin{eqnarray*}
		I^{>,\le} \ll q_m^{4+2\theta-2B} + q_m^{4+2\theta} q_{m+1}^{1-2B} \ll q_m^{-2\theta}.
	\end{eqnarray*}
	
	\subsubsection{$0<|u|\le q_m/4$ and $0<|v|\le q_m/4$}
	In the same vein as the evaluation for \eqref{I1}, we obtain 
	\begin{eqnarray*}
		I^{\le ,\le} 
		&\ll &  
		\sum_{\substack{0<|u_1|, |u_2|\le  q_m/4\\ 0<|v_1|,|v_2| \le  q_m/4 \\ u_1+v_1= u_2+v_2}} |u_1v_1u_2v_2|^{-B} \frac{|1-e(nv_1\alpha)|^2}{ \|u_1\alpha\| \|u_2\alpha\| \|v_1\alpha\|^2} \\
		&& +  \sum_{\substack{0<|u_1|, |u_2|\le  q_m/4\\ 0<|v_1|,|v_2| \le  q_m/4 \\ u_1+v_1= u_2+v_2}} |u_1v_1u_2v_2|^{-B} \frac{|1-e(n(u_1+v_1)\alpha)|^2}{\|u_1\alpha\| \|u_2\alpha\| \|(u_1+v_1)\alpha\|^2}. 
	\end{eqnarray*}
	As before, using $\|u\alpha\|\gg q_m^{-1}$ for $0<|u|<q_m$ and $1-e(n(u+v)\alpha) \ll \max(|u|,|v|)q_m^{\theta}/q_{m+1}$, we have
	\begin{eqnarray*}
		I^{\le ,\le} 
		&\ll &  \frac{q_m^{4+2\theta}}{q_{m+1}^2}
		\sum_{\substack{0<|u_1|, |u_2|\le  q_m/4\\ 0<|v_1|\le |u_1|,\,|v_2| \le  q_m/4 \\ u_1+v_1= u_2+v_2}} |u_1|^{2-B}|v_1u_2v_2|^{-B}\ll \frac{q_m^{4+2\theta}}{q_{m+1}^2}\ll q_m^{-2\theta}.
	\end{eqnarray*}
	This completes the task \eqref{lemt} for Case (i).

	{\bf Case (ii):} $\# \cQ'<\infty$. We have $q_m^{2+2\theta} > q_{m+1}$ for all $m\ge  m_0$, for some $m_0$. Therefore,  for all $l\ge 0$,
	\begin{align*}
		q_{l+1}\ll_{m_0} q_l^{2+2\theta}.
	\end{align*}
	Consider  $q_m\in \cQ'' =\{q_l\in \cQ: l\ge m_0\}$. 
	
	\subsection{Case (ii): Evaluation of $\Sigma_{n,0}$}\label{1c} 
	By \eqref{sn0} and \eqref{0b}, we have $\Sigma_{n,0}= \Sigma_{n,0}^{<} +O(q_m^{-\theta})$ and then we express
	\begin{eqnarray*}
		\Sigma_{n,0}^{<} = \Sigma_{n,0}^{<,1} - n\Sigma_{q_m,0}^{<,2},
	\end{eqnarray*}
	where
	\begin{eqnarray*}
		\Sigma_{n,0}^{<,1} = \sum_{0< |u|< q_m} \frac{\widehat{\varphi}(u)\widehat{\eta}(-u)}{1-e(u\alpha)}  \frac{1-e(-nu\alpha)}{1-e(-u\alpha)} 
		\quad \mbox{ and } \quad
		\Sigma_{q_m,0}^{<,2}:= \sum_{0< |u|< q_m} \frac{\widehat{\varphi}(u)\widehat{\eta}(-u)}{1-e(u\alpha)} .
	\end{eqnarray*}
	We treat $ \Sigma_{n,0}^{<,1}$ with \eqref{eqB} and \eqref{p4} as follows:
	\begin{eqnarray*}
		\Sigma_{n,0}^{<,1} 
		&=& \sum_{0< |u|< q_m} \frac{\widehat{\varphi}(u)\widehat{\eta}(-u)}{1-e(u\alpha)}  \frac{1-e(-nu\alpha)}{1-e(-u\alpha)}\ll \frac{q_m^\theta}{q_{m+1}}\sum_{0< u< q_{m-1}} u^{1-2B}\|u\alpha\|^{-2}
		\\
		&\ll& \frac{q_m^{\theta}}{ q_{m+1}} \sum_{0\le l<m} q_l^{1-2B}\sum_{q_l\le u < q_{l+1}} \|u\alpha\|^{-2}
		\ll \frac{q_m^{\theta}}{ q_{m+1}} \sum_{0\le l<m} q_l^{1-2B} q_{l+1}^2.
	\end{eqnarray*}
	We split $q_{l+1}^2= q_{l+1}^{1-2\theta} q_{l+1}^{1+2\theta}$ and apply $q_{l+1}\le q_m$ and $q_{l+1}\ll q_l^{2+2\theta}$. It follows that
	\begin{eqnarray*}
		\Sigma_{n,0}^{<,1} 
		&\ll&  q_m^{\theta-1} \cdot q_m^{1-2\theta}\sum_{0\le l<m} q_l^{1-2B+(2+2\theta)(1+2\theta)}  \\
		&= & q_m^{-\theta}\sum_{0\le l<m} q_l^{4+6\theta+4\theta^2 -2B-1}
		\ll q_m^{-\theta},
	\end{eqnarray*}
	since $4+6\theta+4\theta^2<4+10\theta<2B$. 
	Hence, we get the first equation in \eqref{lemt}: 
	\begin{eqnarray*}
		\Sigma_{n,0} = -n  \Sigma_{q_m,0}^{<,2} +O(q_m^{-\theta}).
	\end{eqnarray*}
	
	\subsection{Case (ii). Evaluation of $\Sigma_n(t)$}\label{1d} We go through again the evaluation of 
	\begin{eqnarray*}
		\int_{\mathbb{T}}|\Sigma_n(t) |^2\,\textup{d} t\ll \sum_{a,b\in \{>,\leq \}} I^{a,b}, 
	\end{eqnarray*}
	where $I^{a,b}:=\int_{\mathbb{T}}| \Sigma_n^{a,b}(t)|^2\,\textup{d} t$, as in \S\S~\ref{1b}.
	
	\subsubsection{Both $|u|$ and $|v|> q_m/4$} The estimate
	\begin{eqnarray*}
		I^{>,>} \ll q_m^{-1-4\theta}
	\end{eqnarray*}
	is still valid.

	\subsubsection{$0<|u|\le q_m/4$ and $|v|> q_m/4$} We have from \eqref{I1} and  \eqref{S12} that
	\begin{eqnarray*}
		I^{\le ,>}  \ll  n^2 \sum_{\substack{0<|u_1|, |u_2|\le q_m/4\\ |v_1|,|v_2| > q_m/4 \\ u_1+v_1= u_2+v_2}} \frac{|u_1v_1u_2v_2|^{-B}}{|1-e(u_1\alpha)||1-e(u_2\alpha)|} .
	\end{eqnarray*}
	Now we handle more carefully the contribution of $\|u_1\alpha\|$ and $\|u_2\alpha\|$  due to the denominator rather than giving the trivial bound $q_m^{-1}$. Note that $v_2$ is determined (with two possibilities) for given $u_1,u_2,v_1$ and $|v_2| \asymp |v_1|$. Thus, with \eqref{p4},
	\begin{eqnarray*}
		I^{\le ,>}  
		&\ll&  n^2  \sum_{\substack{0<|u_1|, |u_2|\le q_m/4\\ |v_1|,|v_2| > q_m/4 \\ u_1+v_1= u_2+v_2}} |u_1v_1u_2v_2|^{-B}\|u_1\alpha\|^{-1}\|u_2\alpha\|^{-1}\\
		&\ll& q_m^{2+2\theta} \Big(\sum_{0<|u|\le q_m/4} |u|^{-B} \|u\alpha\|^{-1}\Big)^2 \sum_{ |v_1| > q_m/4} |v_1|^{-2B} \\
		&\ll& q_m^{3+2\theta-2B} \Big(\sum_{0\le l< m} q_l^{-B} \sum_{q_l\le u< q_{l+1}}  \|u\alpha\|^{-1}\Big)^2\\
		&\ll& q_m^{3+2\theta-2B} \Big(\sum_{0\le l< m} q_l^{-B} q_{l+1}\log q_{l+1}\Big)^2.
	\end{eqnarray*}
	As $q_{l+1} =q_{l+1}^{\frac12(1-\theta)}q_{l+1}^{\frac12(1+\theta)} \ll q_l^{1-\theta^2} q_m^{\frac12(1+\theta)}< q_lq_m^{\frac12(1+\theta)}$ by $q_{l+1}\ll q_l^{2+2\theta}$ and $q_{l+1}\le q_m$, we get, from the last line,
	\begin{eqnarray*}
		I^{\le ,>} 
		&\ll& q_m^{3+2\theta-2B+ 1+\theta} \Big(\sum_{0\le l< m} q_l^{1-B} \log q_{l+1}\Big)^2.
	\end{eqnarray*}
	As $B>2$ by \eqref{eqBt}, we have $\sum_{q\ge 1} q^{1-B} \log q\ll_B 1$ and thus,
	\begin{eqnarray*}
		I^{\le ,>} \ll q_m^{4+3\theta-2B}\ll q_m^{-\theta}.
	\end{eqnarray*}
	
	\subsubsection{$|u|> q_m/4$ and $0<|v|\le  q_m/4$} We have from the result for $I^{\le ,>}$ that
	\begin{eqnarray*}
		I^{>,\le} \ll  q_m^{-\theta}.
	\end{eqnarray*}
	
	\subsubsection{$0<|u|\le q_m/4$ and $0<|v|\le q_m/4$} Let $g$ be defined as in \eqref{g}. 
	As $|a+b|^2\le 2|a|^2+2|b|^2$, we have 
	\begin{eqnarray*}
		I^{\le ,\le} 
		&=&  \int_{\mathbb{T}}  \bigg|\sum_{\substack{0\neq |u|,|v|\le q_m/4 \\ u+v\neq 0} } \widehat{\varphi}(u)\widehat{\eta}(v)  \sum_{j=1}^{n-1} \sum_{r=0}^{j-1}  e((ur+vj)\alpha) e((u+v)t)\bigg|^2\,\textup{d} t \\
		&\le& 8 \int_{\mathbb{T}}  \bigg|\sum_{0\neq |v|< |u|\le q_m/4 } \widehat{\varphi}(u)\widehat{\eta}(v)   g(u,v)  e((u+v)t) \bigg|^2\,\textup{d} t  
		\\
		&& \ + 2 \int_{\mathbb{T}}  \bigg|\sum_{\substack{0\neq|v|= |u|\le q_m/4 \\ u+v\neq 0} } \widehat{\varphi}(u)\widehat{\eta}(u) 
		g(u,u) e((u+v)t) \bigg|^2\,\textup{d} t .
	\end{eqnarray*}
	In view of the expression of $g(u,v)$ (in \eqref{g}), the first integral is
	\begin{eqnarray*}
		&\ll& 
		\int_{\mathbb{T}}  \bigg|\sum_{0\neq |v|< |u|\le q_m/4 } \widehat{\varphi}(u)\widehat{\eta}(v)  \frac{1-e(nu\alpha)}{(1-e(v\alpha))(1-e(u\alpha))} e(nv\alpha)  e((u+v)t)\bigg|^2\,\textup{d} t \\
		&& + 
		\int_{\mathbb{T}}  \bigg|\sum_{0\neq |v|< |u|\le q_m/4 } \widehat{\varphi}(u)\widehat{\eta}(v)  \frac{1-e(n(u+v)\alpha)}{(1-e(v\alpha))(1-e((u+v)\alpha))} e(v\alpha) e((u+v)t)\bigg|^2\,\textup{d} t \\
		&=:& I_{1,1} + I_{1,2},
	\end{eqnarray*}
	and similarly, the second integral is
	\begin{eqnarray*}
		&\ll& 
		\int_{\mathbb{T}}  \bigg|\sum_{\substack{0<|v|= |u|\le q_m/4\\ u+v\neq 0} } \widehat{\varphi}(u)\widehat{\eta}(v)  \frac{1-e(nv\alpha)}{(1-e(v\alpha))(1-e(u\alpha))}  e(nv\alpha)  e((u+v)t)\bigg|^2\,\textup{d} t \\
		&& + 
		\int_{\mathbb{T}}  \bigg|\sum_{\substack{0< |v|= |u|\le q_m/4\\ u+v\neq 0} } \widehat{\varphi}(u)\widehat{\eta}(v)  \frac{1-e(n(u+v)\alpha)}{(1-e( v\alpha))(1-e((u+v)\alpha))} e(v\alpha) 
		e((u+v)t)\bigg|^2\,\textup{d} t\\
		&=:& I_{2,1} + I_{2,2}. 
	\end{eqnarray*}
	Consequently, 
	\begin{eqnarray*}
		I^{\le ,\le} \ll (I_{1,1} + I_{2,1}) + (I_{1,2}+ I_{2,2}).
	\end{eqnarray*}
	We shall show that both brackets are $\ll q_m^{-\theta}$ in the following two subsections, which will then complete the proof of Lemma~\ref{lem31}. 
	
	\subsubsection*{\rm 6.4.4 (I): Treatment of $I_{1,1}+I_{2,1}$.} 
	To handle $I_{1,1}$, we convert the double sum over $u$ and $v$ into
	\begin{eqnarray*}
		\sum_{0\le a\le b< m} \sum_{\substack{q_a\le |v| <\min (q_{a+1},q_m/4)\\ q_b\le |u| <\min (q_{b+1},q_m/4)\\ |v|<|u|}}.
	\end{eqnarray*}
	We can treat $I_{2,1}$ in the same fashion. 
	Squaring out the integrand and integrating over $t$, we infer with $1-e(u\alpha) \asymp \|u\alpha\|$, $1-e(v\alpha) \asymp \|v\alpha\|$ and 
	$1-e(nu\alpha) \ll |u| q_m^\theta/q_{m+1}$ that
	\begin{eqnarray*}
		I_{1,1} +I_{2,1} \ll \frac{q_m^{2\theta}}{q_{m+1}^2} \sum_{\substack{0\le a\le b< m\\ 0\le c\le d< m}} 
		\sum_{\substack{q_a\le |v| < q_{a+1}\\ q_b\le |u| < q_{b+1}\\ q_c\le |v'| < q_{c+1}\\ q_d\le |u'| < q_{d+1}\\ |v|\le |u|, |v'|\le |u'|\\ 0\neq u+v= u'+v'}} |uu'vv'|^{-B} |uu'| \|u\alpha\|^{-1}\|u'\alpha\|^{-1}\|v\alpha\|^{-1}\|v'\alpha\|^{-1}.
	\end{eqnarray*}
	By symmetry, we may assume $b\le d$ and make several uses of \eqref{p4} and \eqref{p2}, which give 
	\begin{itemize}
		\item $\sum_{0<|u|< q_k} \|u\alpha\|^{-1} \ll q_k \log (q_k+1)\ll q_k \log q_m$ for $k < m$,
		\item $\|q\alpha\|^{-1} \le \|q_k\alpha\|^{-1}\ll q_{k+1}$ for $0<q<q_{k+1}$.
	\end{itemize} 
	To analyze the multiple sums, we separate into the following cases:
	\begin{enumerate}[leftmargin=7mm]
		\item[(i)] $b<d$: The multiple sum is evaluated as follows: 
		\begin{eqnarray*}
			&& \sum_{\substack{0\le a\le b< d< m\\ 0\le c\le d}} 
			\sum_{\substack{q_a\le |v| < q_{a+1}\\ q_b\le |u| < q_{b+1}\\ q_c\le |v'| < q_{c+1}\\ q_d\le |u'| < q_{d+1}\\ |v|\le |u|, |v'|\le |u'|\\ 0\neq u+v= u'+v'}} |uu'vv'|^{-B} |uu'| \|u\alpha\|^{-1}\|u'\alpha\|^{-1}\|v\alpha\|^{-1}\|v'\alpha\|^{-1} \\
			&\ll& 
			\sum_{\substack{0\le a\le b <d< m\\ 0\le c\le d}} 
			\sum_{\substack{q_a\le |v| < q_{a+1}\\ q_b\le |u| < q_{b+1}\\ q_c\le |v'| < q_{c+1}\\ q_d\le |u'| < q_{d+1}\\ |v|\le |u|, |v'|\le |u'|\\ 0\neq u+v= u'+v'}} 
			|uvv'|^{-B} |u| \|u\alpha\|^{-1}\|v\alpha\|^{-1}\|v'\alpha\|^{-1} q_{a+1}^{-1/2}q_{b+1}^{-1}q_{c+1}^{-1/2}\\
			&& \times q_d^{3/2+10\theta} q_{d+1}^{1/2-4\theta} |u'|^{1-B}\|u'\alpha\|^{-1}
		\end{eqnarray*}
		as $q_{a+1}, q_{b+1} \le q_d$, $q_{c+1} \le q_{d+1}$ and $q_{d+1}^{4\theta} \ll q_d^{10\theta}$ (by $q_{l+1}\ll q_l^{2+2\theta}$). 
		
		Due to $u+v= u'+v'$, the value of $u'$ (and hence $q_d$) is determined for given $u,v,v'$. As $|u'|^{1-B} \|u'\alpha\|^{-1}\ll q_d^{1-B} q_{d+1}$ and   
		$$
		q_d^{3/2+10\theta} q_{d+1}^{1/2-4\theta} q_d^{1-B} q_{d+1}= q_d^{5/2+10\theta-B}q_{d+1}^{3/2-4\theta} \le q_d^{1/2} q_{d+1}^{3/2-4\theta} \le q_m^{1/2} q_m^{3/2-4\theta}$$ 
		for $B> 2+10\theta$ (see \eqref{eqBt}) and $d<m$, we see that the multiple sum is 
		\begin{eqnarray*} 
			&\ll& 
			q_m^{2-4\theta}
			\sum_{\substack{0\le a\le b< m\\ 0\le c< m}} 
			\sum_{\substack{q_a\le |v| < q_{a+1}\\ q_b\le |u| < q_{b+1}\\ q_c\le |v'| < q_{c+1}}} |uvv'|^{-B} |u| \|u\alpha\|^{-1}\|v\alpha\|^{-1}\|v'\alpha\|^{-1} q_{a+1}^{-1/2}q_{b+1}^{-1}q_{c+1}^{-1/2}\\
			&\ll& 
			q_m^{2-4\theta}
			\sum_{\substack{0\le a\le b< m\\ 0\le c< m}} q_b^{1-B} q_a^{-B}q_c^{-B}
			\sum_{\substack{q_a\le |v| < q_{a+1}\\ q_b\le |u| < q_{b+1}\\ q_c\le |v'| < q_{c+1}}}  \|u\alpha\|^{-1}\|v\alpha\|^{-1}\|v'\alpha\|^{-1} q_{a+1}^{-1/2}q_{b+1}^{-1}q_{c+1}^{-1/2}\\
			&\ll& 
			q_m^{2-4\theta}\log^3 q_m
			\sum_{\substack{0\le a\le b< m\\ 0\le c< m}} q_b^{1-B} q_a^{-B}q_{a+1}^{1/2}q_c^{-B}q_{c+1}^{1/2} \\
			&\ll& q_m^{2-4\theta}\log^3 q_m
			\sum_{\substack{0\le a\le b< m\\ 0\le c< m}} q_b^{1-B} q_a^{1+\theta-B}q_c^{1+\theta-B}
			\ll     q_m^{2-4\theta}\log^3 q_m,
		\end{eqnarray*}
		by $q_{l+1}\ll q_l^{2+2\theta}$ and $1+\theta -B<-1$.
		
		\item[(ii)] $b=d$: We split the multiple sum into three pieces:
		\begin{eqnarray*}
			&& \sum_{\substack{0\le a\le b< m\\ 0\le c\le b}} 
			\sum_{\substack{q_a\le |v| < q_{a+1}\\ q_b\le |u| < q_{b+1}\\ q_c\le |v'| < q_{c+1}\\ q_b\le |u'| < q_{b+1}\\ |v|\le |u|, |v'|\le |u'|\\ 0\neq u+v= u'+v'}} |uu'vv'|^{-B} |uu'| \|u\alpha\|^{-1}\|u'\alpha\|^{-1}\|v\alpha\|^{-1}\|v'\alpha\|^{-1}\\
			&=& \sum_{b< m}\sum_{\substack{a,c< b}} + \sum_{a=b< m}\sum_{\substack{c\le b}} + \sum_{c=b< m}\sum_{\substack{a< b}} .
		\end{eqnarray*}
		For the first two pieces, we vary $u,v,v'$ and determine $u'$ by $u+v=u'+v'$. Thus,
		\begin{eqnarray*}
			\sum_{b< m}
			\sum_{\substack{a,c< b}} \cdots &\ll& \sum_{b< m}
			\sum_{\substack{a,c< b}} 
			\sum_{\substack{q_a\le |v| < q_{a+1}\\ q_b\le |u| < q_{b+1}\\ q_c\le |v'| < q_{c+1}}} |uvv'|^{-B} |u| \|u\alpha\|^{-1}\|v\alpha\|^{-1}\|v'\alpha\|^{-1} q_b^{1-B} q_{b+1}\\
			&\ll& \sum_{b< m}
			\sum_{\substack{a,c< b}} q_b^{1-B} q_{b+1} q_a^{-B} q_c^{-B} q_b^{1-B}
			\sum_{\substack{q_a\le |v| < q_{a+1}\\ q_b\le |u| < q_{b+1}\\ q_c\le |v'| < q_{c+1}}}  \|u\alpha\|^{-1}\|v\alpha\|^{-1}\|v'\alpha\|^{-1} \\
			&\ll& \log^3 q_m\sum_{b< m}
			\sum_{\substack{a,c< b}} q_b^{1-B} q_{b+1}  q_a^{-B} q_{a+1} q_c^{-B} q_{c+1} q_b^{1-B} q_{b+1}
			\\
			&\ll& q_m^\theta\sum_{b< m}
			q_b^{2-2B} q_{b+1}^2 \bigg(\sum_{\substack{a< b}} q_a^{-B} q_{a+1}\bigg)^2
			\\
			&\ll& q_m^\theta \sum_{b< m} q_b^{3-2B} q_{b+1}^2  \ll  q_m^\theta\sum_{b< m} q_b^{3+10\theta-2B} q_{m}^{2-4\theta} \\
			&\ll& q_m^{2-3\theta},
		\end{eqnarray*}
		because  $\sum_{\substack{a< b}} q_a^{-B} q_{a+1}\le q_b^{1/2}\sum_{\substack{a< b}} q_a^{-B} q_{a+1}^{1/2} \ll q_b^{1/2}\sum_{\substack{a< b}} q_a^{-B+(2+2\theta)/2}\ll q_b^{1/2}$,
		and 
		\begin{eqnarray*}
			\sum_{a=b< m}
			\sum_{\substack{c\le  b}} \cdots
			&\ll& \sum_{b< m}
			\sum_{\substack{c\le b}} 
			\sum_{\substack{q_b\le |v| < q_{b+1}\\ q_b\le |u| < q_{b+1}\\ q_c\le |v'| < q_{c+1}}} |uvv'|^{-B} |u| \|u\alpha\|^{-1}\|v\alpha\|^{-1}\|v'\alpha\|^{-1} q_b^{1-B} q_{b+1}\\
			&\ll& \log^3 q_m\sum_{b< m}
			\sum_{\substack{c\le b}} q_b^{1-B} q_{b+1}  q_b^{-B} q_{b+1} q_c^{-B} q_{c+1} q_b^{1-B} q_{b+1}
			\\
			&\ll&  q_m^\theta \sum_{b< m}
			q_b^{2-3B} q_{b+1}^3 \sum_{\substack{c\le b}}  q_c^{1+\theta -B} q_{c+1}^{1/2}
			\\
			&\ll& q_m^\theta \sum_{b< m} q_b^{5+13\theta-3B} q_{b+1}^{2-4\theta} \sum_{\substack{c \le b}}  q_c^{1+\theta -B} \ll  q_m^{2-3\theta},
		\end{eqnarray*}
		because $q_{c+1}^{1/2} \le q_c^{1+\theta}$, $q_{c+1}^{1/2} \le q_{b+1}^{1/2} \le q_b^{1+\theta}$ and $q_{b+1}^{1+4\theta} \ll q_b^{(2+2\theta)(1+4\theta)}\le q_b^{2+12\theta}$.
		
		For the third piece, we vary $v,u',v'$ and determine $u$ by $u+v=u'+v'$. Switching the roles of $u$ and $u'$ and $a$ and $c$, this case is included in the evaluation of the second piece.
	\end{enumerate}
	We conclude that 
	\begin{eqnarray*}
		I_{1,1} +I_{2,1} \ll \frac{q_m^{2\theta}}{q_{m+1}^2} \big(q_m^{2-4\theta}\log^3 q_m + q_m^{2-3\theta}\big) \ll q_m^{-\theta}.
	\end{eqnarray*}
	\subsubsection*{\rm \rm 6.4.4 (II): Treatment of $I_{1,2}+I_{2,2}$.}  
	Next we deal with $I_{1,2}$ and $I_{2,2}$:
	\begin{eqnarray*}
		I_{1,2}
		&\ll&  
		\int_{\mathbb{T}}  \bigg|\sum_{0\neq |v|< |u|\le q_m/4 } \widehat{\varphi}(u)\widehat{\eta}(v)  \frac{1-e(n(u+v)\alpha)}{(1-e(v\alpha))(1-e((u+v)\alpha))} 
		e(v\alpha) e((u+v)t)\bigg|^2\,\textup{d} t\\
		&\ll&  
		\int_{\mathbb{T}}  \bigg|\sum_{\substack{0\neq |v|< |u|\le q_m/4 \\ uv>0} }\cdots\bigg|^2\,\textup{d} t
		+ 
		\int_{\mathbb{T}}  \bigg|\sum_{\substack{0\neq |v|< |u|\le q_m/4 \\ uv<0} }\cdots\bigg|^2\,\textup{d} t.
	\end{eqnarray*}
	Treat $I_{2,2}$ similarly. Applying $1-e(n(u+v)\alpha)\ll q_m^\theta (|u+v|)/q_{m+1}$, we deduce that 
	\begin{eqnarray*}
		I_{1,2} + I_{2,2}\ll \frac{q_m^{2\theta}}{q_{m+1}^2} \big(J_+ + J_-\big),
	\end{eqnarray*}
	where, after replacing $|u|=\pm u$ with $u\in \N$, etc., 
	\begin{align*}
		J_+ &:=  \sum_{\substack{0<v \le  u\le q_m/4\\ 0<v'  \le  u'\le q_m/4\\ u+v=u'+v'}}
		(uu'vv')^{-B} (u+v)(u'+v') \|(u+v)\alpha\|^{-1}\|(u'+v')\alpha\|^{-1}\|v\alpha\|^{-1}\|v'\alpha\|^{-1},\\    
		J_- &:= \sum_{\substack{0<v< u\le q_m/4\\ 0<v'< u'\le q_m/4\\ u-v=u'-v'}}
		(uu'vv')^{-B} (u-v)(u'-v') \|(u-v)\alpha\|^{-1}\|(u'-v')\alpha\|^{-1}\|v\alpha\|^{-1}\|v'\alpha\|^{-1}.
	\end{align*}
	(Note that $I_{2,2}$ does not contribute to $J_-$.)
	Our proof will be complete by showing 
	\begin{eqnarray}\label{Jpm}
		J_++ J_- \ll q_m^{2-3\theta}.
	\end{eqnarray}
	To its end we let
	\begin{eqnarray*}
		K:= \sum_{\substack{0<w< q_m\\ 0<v\le v'\le q_m/4}} \|w\alpha\|^{-2} w^{2-2B} 
		(vv')^{-B} \|v\alpha\|^{-1}\|v'\alpha\|^{-1}.
	\end{eqnarray*}
	
	Let $w=u+v \  (=u'+v')\asymp u$ when $0<v  \le u$  ($0<v' \le u'$). After rearranging, we see that
	\begin{eqnarray*} 
		J_+
		&\ll&  \sum_{0<w\le q_m/2} \|w\alpha\|^{-2} w^{2-2B} 
		\sum_{0<v,v' \le q_m/4} (vv')^{-B} \|v\alpha\|^{-1}\|v'\alpha\|^{-1}\ll K. 
	\end{eqnarray*}  
	By symmetry, we may assume $v\le v'$ in $J_-$ and write $w=u-v \ (=u'-v')$. Then we have
	\begin{eqnarray*}
		J_-
		&=& \sum_{\substack{0<v < u\le q_m/4\\ 0<v' < u'\le q_m/4\\ u-v=u'-v'}}
		|uu'vv'|^{-B} |u-v||u'-v'| \|(u-v)\alpha\|^{-1}\|(u'-v')\alpha\|^{-1}\|v\alpha\|^{-1}\|v'\alpha\|^{-1}\\
		&\ll&  \sum_{\substack{0<w< q_m\\ 0<v<v'\le q_m/4}} \|w\alpha\|^{-2} w^2 
		((w+v)(w+v')vv')^{-B} \|v\alpha\|^{-1}\|v'\alpha\|^{-1}  \\
		&\ll&  \sum_{\substack{0<w< q_m\\ 0<v<v'\le q_m/4}} \|w\alpha\|^{-2} w^{2-2B} 
		(vv')^{-B} \|v\alpha\|^{-1}\|v'\alpha\|^{-1}  \\
		&\ll& K.
	\end{eqnarray*}
	It remains to show $K\ll q_m^{2-3\theta}$. We have
	\begin{eqnarray*}
		K   &\ll&  \sum_{\substack{a\le b<m\\ c<m}} \sum_{\substack{q_a\le v< q_{a+1}\\ q_b\le v'< q_{b+1} \\ q_c\le w< q_{c+1}}} \|w\alpha\|^{-2} w^{2-2B} (vv')^{-B} \|v\alpha\|^{-1}\|v'\alpha\|^{-1}  \\
		&\ll&  \sum_{\substack{a\le b<m\\ c<m}} q_c^{2-2B} (q_aq_b)^{-B} 
		\sum_{\substack{q_a\le v< q_{a+1}\\ q_b\le v'< q_{b+1} \\ q_c\le w< q_{c+1}}} \|w\alpha\|^{-2} \|v\alpha\|^{-1}\|v'\alpha\|^{-1}  \\
		&\ll&  \log^2 q_m\sum_{\substack{a\le b<m\\ c<m}} q_c^{2-2B} q_{c+1}^2 q_a^{-B} q_{a+1}q_b^{-B} q_{b+1} \\   
		&\ll& q_m^\theta (K_1+K_2),
	\end{eqnarray*}
	where
	\begin{eqnarray*}
		K_1 & =&  \sum_{\substack{a\le b<m\\ c\le b}}q_c^{2-2B} q_{c+1}^2 q_a^{-B} q_{a+1}q_b^{-B} q_{b+1}, \\
		K_2 &=& \sum_{a\le b<c<m}q_c^{2-2B} q_{c+1}^2 q_a^{-B} q_{a+1}q_b^{-B} q_{b+1}.
	\end{eqnarray*}
	Invoking $q_{l+1} \ll q_l^{2+2\theta} <q_l^B$ and noting from \eqref{eqBt} that  $1+\theta -B  < 1+4\theta -B <-1$, $3+\theta-2B<-1$ and $9/2+12\theta -3B<-1$,
	we see that
	\begin{eqnarray*}
		K_1&=&\sum_{\substack{a\le b<m\\ c\le b}}q_c^{2-2B} q_{c+1}^2 q_a^{-B} q_{a+1}q_b^{-B} q_{b+1} \\    
		&\ll & \sum_{b<m}q_b^{-B} q_{b+1} \bigg(\sum_{a<b} q_a^{1+\theta-B} q_{a+1}^{1/2} + q_b^{-B} q_{b+1}\bigg)\bigg(\sum_{c<b} q_c^{3+\theta-2B} q_{c+1}^{3/2} + q_b^{2-2B} q_{b+1}^2\bigg)\\
		&\ll& \sum_{b<m}q_b^{-B} q_{b+1} \big(q_b^{1/2} + q_b^{-B}q_{b+1}\big) \big(q_b^{3/2} + q_b^{2-2B}q_{b+1}^2\big)\\
		&\ll&\sum_{b<m}q_b^{-B} q_{b+1} q_b^{1/2}\big(q_b^{3/2}+q_b^{2-2B}q_{b+1}^2\big)
		\ll \sum_{b<m}q_b^{2-B} q_{b+1} + \sum_{b<m}q_b^{5/2-3B} q_{b+1}^3 \\
		&\ll& q_m\sum_{b<m}q_b^{1+4\theta-B}q_b^{1-4\theta}  + q_m^{2-4\theta}\sum_{b<m}q_b^{5/2-3B} q_{b+1}^{1+4\theta} \\
		&\ll& q_m^{2-4\theta} \sum_{b<m}q_b^{1+4\theta-B} + q_m^{2-4\theta}\sum_{b<m}q_b^{5/2-3B} q_b^{2+12\theta} \\
		&\ll& q_m^{2-4\theta}. 
	\end{eqnarray*}
	Next,  recalling  $1+\theta -B<-1$ and noting $3+10\theta -2B<-1$, we have
	\begin{eqnarray*}
		K_2&=&\sum_{a\le b<c<m}q_c^{2-2B} q_{c+1}^2 q_a^{-B} q_{a+1}q_b^{-B} q_{b+1}\\     
		&\ll & \sum_{c<m}q_c^{2-2B} q_{c+1}^2 \bigg(\sum_{a<c} q_a^{-B} q_{a+1}\bigg)^2
		\ll  \sum_{c<m}q_c^{2-2B} q_{c+1}^2 \bigg(\sum_{a<c} q_a^{1+\theta-B} q_{a+1}^{1/2}\bigg)^2\\
		&\ll & \sum_{c<m}q_c^{3-2B} q_{c+1}^2 \bigg(\sum_{a<c} q_a^{1+\theta-B} \bigg)^2 \ll  \sum_{c<m}q_c^{3-2B} q_{c+1}^2 \\
		&\ll & \sum_{c<m}q_c^{3+10\theta-2B} q_{c+1}^{2-4\theta} \ll  q_m^{2-4\theta}\sum_{c<m}q_c^{3+10\theta-2B} \\
		&\ll& q_m^{2-4\theta}. 
	\end{eqnarray*}
	We conclude that $K\ll q_m^{2-3\theta}$ and hence  $J_+, J_- \ll q_m^{2-3\theta}$, i.e. \eqref{Jpm} holds. 
	Our proof of Lemma~\ref{lem31} is complete.
	
	\section{Proof of Theorem~\ref{thm}}
	
	For irrational $\alpha$, the M\"obius disjointness follows from Theorem~\ref{intd} and Proposition~\ref{prop2.1}. It remains to consider the case $\alpha\in \Q\cap [0,1)$. 
	This case has been handled by Ma and Wu in \cite[Theorem 1.1]{MW} under the assumption that $\varphi, \eta, \psi$ are $C^\infty$. The smoothness condition can be relaxed, for example, He and Wang \cite{HW} considered  $C^{1+\varepsilon}$ functions but also assumed $\varphi= \eta$. 
	
	Below we provide a brief verification following \cite{MW} with a little modification. By \cite[Lemma 2.1]{MW}, which is \cite[Proposition 1]{HLW}, it suffices to show 
	\begin{eqnarray}\label{smo}
		\sum_{n\le N} \mu(n) f(S_\alpha^n (t_{0}, \overline{g_0})) = o(N)
	\end{eqnarray}
	for $f\in \mathcal{A}\cup \mathcal{B}$. The families $\mathcal{A}$ and $\mathcal{B}$ are defined as in Section~2 of \cite{HLW} or \cite{MW}, from which we see that $\mathcal{A}$ consists of scalar multiples of functions in $\mathcal{A}'$ and a function $f\in \mathcal{A}'$ is of the form
	\begin{eqnarray*}
		&& f(t,\overline{g}) = e(m_0 t+m_1 x+m_2 y+m z) \sum_{b \in r+\Z} e^{-\pi (y + b)^2-\pi \delta (y+b)} e(b mx) 
		\ \mbox{ if } \
		g =\begin{pmatrix} \begin{smallmatrix} 1 & y & z \\ & 1 & x\\ & & 1 \end{smallmatrix}\end{pmatrix},
	\end{eqnarray*}
	where $m_0, m_1,m_2,  m\in \Z$, $r\in \Q$ and $\delta =0$ or $1 \pm i$. Moreover, by the definition of $\mathcal{B}$ (or from the proof of \cite[Proposition 2]{HLW}), the family $\mathcal{B}$ can be identified with $C(\mathbb{T})\otimes C(\mathbb{T}^2)$. As characters are linearly dense in ${C}(\mathbb{T}^k)$, $k=1,2$, we may reduce $\mathcal{B}$ to its subfamily $\mathcal{B}'$ which consists of $f\in \mathcal{B}$ of the form $f(t,\overline{g}) = e(m_0 t) e(m_1 x) e(m_2 y)$ with $m_0,m_1,m_2\in \Z$. 
	
	By \eqref{Sn}, it follows that for $  g_0 =\begin{pmatrix} \begin{smallmatrix} 1 & y_0 & z_0 \\ & 1 & x_0\\ & & 1 \end{smallmatrix}\end{pmatrix}$,
	\begin{eqnarray*}
		S_\alpha^n (t_{0}, \overline{g_0}) = 
		\left(t_{0} + n\alpha, \Gamma \begin{pmatrix}\begin{smallmatrix}  1 \ & y_0+\Phi_n(t) \,  & z_0 + y_{0} \xi_n(t)+ \Omega_n(t) + H_n(t) \\ & 1 & x_0+\xi_n(t)  \\ & & 1 \end{smallmatrix} \end{pmatrix}\right). 
	\end{eqnarray*}
	For any rational $\alpha = a/q\in [0,1)$ (with $(a,q)=1$) and function $h$ of period 1, we have
	$$
	C_{h}(n,t) :=\sum_{r=0}^{n-1} h(t+r\alpha) = q^{-1}C_{h}(q,t) \cdot (n-b) + C_{h}(b,t)
	$$
	if  $n= b'q+b$ where $b' \geq 0$ and $0\le b< q$ (or see \cite[p.305]{MW}). That $C_{h}(n,t)$ is a (linear) polynomial in $n$ is the key upshot. We shall apply it to $\xi_n(t)= C_\eta(n,t) $, $\Phi_n(t)= C_\varphi(n,t)$ and $\Omega_n(t)= C_\psi(n,t)$. The case of $H_n(t)$ is more complicated and we have
	\begin{eqnarray*}
		&& H_n(t)=\sum_{j=1}^{n-1} \sum_{r=0}^{j-1} \varphi (t+r\alpha) \eta(t+j\alpha) \\
		&=& \bigg(\sum_{s=0}^{b'-1} \sum_{j=sq+1}^{sq+q} + \sum_{j= b'q+1}^{b'q+b-1}\bigg)    \eta(t+j\alpha)\sum_{r=0}^{j-1} \varphi (t+r\alpha)
		\\
		&=& \sum_{s=0}^{b'-1} \sum_{j=1}^{q} \eta(t+j\alpha) \sum_{r=0}^{sq+j-1} \varphi (t+r\alpha)
		+ \sum_{j= 1}^{b-1}\eta(t+j\alpha) \sum_{r=0}^{b'q+j-1} \varphi (t+r\alpha)\\
		&=&
		\sum_{s=0}^{b'-1}  \sum_{j= 1}^{q}\eta(t+j\alpha) \Big( s C_\varphi(q,t) + C_\varphi(j,t)\Big) +\sum_{j=1}^{b-1}\eta(t+j\alpha) \Big( b' C_\varphi(q,t) + C_\varphi(j,t)\Big)\\
		&=&
		C_\eta(q,t+\alpha) C_\varphi(q,t) \sum_{s=0}^{b'-1}  s  \\
		& & + b' \sum_{j=1}^q \eta(t+j\alpha) C_\varphi(j,t) +b'\Big(C_\eta(b,t) -\eta(t)\Big) C_\varphi(q,t) + \sum_{j=1}^{b-1}\eta(t+j\alpha)  C_\varphi(j,t)\\
		&=&
		\frac12 q^{-2} C_\eta(q,t+\alpha) C_\varphi(q,t) \cdot \big(n^2-(q+2b)n+(q+b)b\big)  \\
		& & +q^{-1}\Big(\sum_{j=1}^{q-1} \eta(t+j\alpha) C_\varphi(j,t) +C_\eta(b,t) C_\varphi(q,t)\Big) \cdot (n-b)+ \sum_{j=1}^{b-1}\eta(t+j\alpha)  C_\varphi(j,t).
	\end{eqnarray*}
	Thus, $H_n(t)$ is a quadratic polynomial in $n$. Consequently,  for $f\in \mathcal{A}'$, we may express
	\begin{eqnarray*}
		&& f(S_\alpha^n (t_{0}, \overline{g_0})) 
		= e(P_2(n)) \sum_{b\in r+\Z} e^{-\pi(b+An+B)^2 - \pi \delta(b+An+B)} e(b(Cn+D))  \\
		&=& e(\{ P_2(n)\}-\lfloor An+B\rfloor (Cn+D)) \sum_{b\in r+\Z} e^{-\pi(b+\{An+B\})^2 - \pi \delta(b+\{An+B\})} e(\{b(Cn+D)\}) 
	\end{eqnarray*}
	and for $f\in \mathcal{B}'$, $ f(S_\alpha^n (t_0, \overline{g_0})) = e(\{P_1(n)\})$,  
	where $A$, $B$, $C$ and $D$ are real numbers independent of $n$ and $b$ while $P_j(n)$ are real polynomials in $n$ of degree $j$, for $j=1,2$.  
	Noticing that the case $f\in \mathcal{B}'$ is simpler than the case $f\in \mathcal{A}'$, we will only deal with the case  $f\in\mathcal{A}'$ below.

	Clearly for any arbitrarily small $\varepsilon\in (0,1)$, the tail part of $|b-r|\ge \sqrt{|\log \varepsilon|}$ is 
	$$
	\ll \sum_{|b| \ge \sqrt{|\log \varepsilon|}} e^{-\pi b^2/2} \ll \varepsilon.
	$$ 
	Splitting the sum  and applying the formula $\int_{-\infty}^\infty e^{-\pi w^2} e(k w)\,dw = e^{-\pi k^2}$, we obtain  
	\begin{eqnarray*}
		f(S_\alpha^n (t_0, \overline{g_0})) 
		=  e^{\delta^2\pi/4} \int_{-\infty}^\infty 
		e^{-\pi w^2} \sum_{|b|< \sqrt{|\log \varepsilon|}} e\big( (b+Q_1(n)+\tfrac{\delta}2)w + Q_2(n))\,dw   +O(\varepsilon),
	\end{eqnarray*}
	where  $Q_1(n)= \{An+B\}$ and $Q_2(n)=\{ P_2(n)\}-\lfloor An+B\rfloor (Cn+D)+ \{b(Cn+D)\}$. Applying the fast decay of $e^{-\pi w^2}$, we split the integral into two pieces according to $|w|\le |\log \varepsilon|$ or not. We bound the latter part trivially and handle the former part with the mean value theorem for integrals. Consequently, we are led to treat 
	\begin{eqnarray*}
		f(S_\alpha^n (t_0, \overline{g_0})) 
		=  2|\log \varepsilon| e^{\delta^2\pi/4} 
		e^{-\pi w_0^2} \sum_{|b|< \sqrt{|\log \varepsilon|}} e\big( (b+Q_1(n)+\tfrac{\delta}2)w_0 + Q_2(n))  +O(\varepsilon)
	\end{eqnarray*}
	for some $w_0\in [-|\log\varepsilon|, |\log\varepsilon|]$. These $Q_j$'s are bracket polynomials defined as in Green-Tao \cite[\S 5]{GT}. Apply the estimate in \cite[p.559]{GT} to $F(\{x\}) = e(\{x\})=e(x)$ and $p(n) = w_0 Q_1(n)+Q_2(n)$, we infer that
	\begin{eqnarray*}
		\frac1N \sum_{n\le N} \mu (n) e(p(n)) = \mathbb{E}_{n\in [N]} \mu(n) F(\{p(n)\}) \ll_{A,\varepsilon} \log^{-A}N, \quad  \forall \ A >1.
	\end{eqnarray*}
	This implies that the left side of \eqref{smo} is $\ll \varepsilon N$ for all sufficiently large $N$, and hence \eqref{smo}  follows.

	\section{Initial steps for the proof of Theorem~\ref{prop}}\label{s8}
	
	Following the argument in \cite{HLW}, we verify the sub-polynomial property of the measure complexity (see \eqref{sp})  as follows: Let  $\varepsilon>0$ be arbitrarily small, $L>\varepsilon^{-1}$ a large enough number and still write $\cQ=\{q_k\}_{k\ge 0}$ as in \S\S~\ref{qk}. We mesh $\T\times \Gamma\backslash G$ with the family  of grid points\footnote{The choice here is different from \cite{HLW} and \cite{MW}.},
	\begin{eqnarray*}
		F_\varepsilon (k)=\Big\{\Big(\frac{j\varepsilon}{q_k^2 L},\overline{\begin{pmatrix}\begin{smallmatrix}    
					1 & \frac{j_2}{q_k L} & \frac{j_3}{q_k L} \\ & 1 & \frac{j_1}{q_k L} \\ & & 1 \end{smallmatrix}\end{pmatrix}} \Big): 0\le j\le \frac{q_k^2 L}\varepsilon -1, 0\le j_1,j_2,j_3 \le q_k L -1\Big\}.
	\end{eqnarray*}
	Later on we shall choose a suitable infinite set of $n$'s, with $n$ explicitly in terms of $q_k$,  such that each grid box can lie in a $\bar{d}_n$-open ball of diameter $\varepsilon$. (See \eqref{bard} for $\bar{d}_n$.) As these grid boxes cover the whole space, their total (probability) measure is 1. In view of \eqref{mc}, 
	\begin{equation}\label{eq8.1}
		s_n({\rm X}, S_{\alpha
		}, d, \nu,\varepsilon) \le \#F_\varepsilon (k) \le \frac{q_k^2 L}{\varepsilon} (q_kL )^3 = \varepsilon^{-1} q_k^5 L^4.
	\end{equation}
	The $n$'s will be chosen so that $\varepsilon^{-1} q_k^5 L^4 \ll n^\tau$, $\forall$ $\tau> 0$.  
	
	We separate into two cases according as 
	\begin{equation}\label{eqqB}
		\cQ_B:=\{ q_\ell^{B-1}: q_\ell\in \cQ \mbox{ satisfies } q_\ell^B< q_{\ell+1}\}
	\end{equation}
	is infinite or not. The value of $B$ will be specified later.

	As in \cite{HLW} and \cite{MW}, we will construct a map $R:(\T\times \Gamma\backslash G, S_{\alpha})\to (\T\times \Gamma\backslash G, S_{\alpha})$  and show a sub-polynomial measure complexity of the flow associated with $R^{-1} \circ S_\alpha \circ R $ in lieu of $S_\alpha$.
	
	\subsection{Construction of $R$}\label{ss8.1}
	
	We start with the decomposition of the smooth function $\phi$ on $\R/\Z$ into resonant part $\phi^+$ and non-resonant part $\phi^-$ as in \cite{HLW}.
	Explicitly, we fix any $B> 2$ and set
	\begin{eqnarray*}
		M_1(B)&=& \bigcup_{q_k\in \cQ\atop 1< q_k< q_{k+1}^{1/B}} \{m\in \Z: q_k\le |m|< q_{k+1}, \ q_k\big|m\} \cup \{0\},\\
		M_2(B)&=& \Z\setminus M_1(B) \\
		&=& \bigcup_{q_k\in \cQ\atop 1< q_k< q_{k+1}^{1/B}} \{m\in \Z: q_k\le |m|< q_{k+1}, q_k\!\not\big|\,m\}  \\
		& &\mbox{ }\cup \bigg(\{m\in \Z: 1\le |m|< q_1\} \cup \bigcup_{q_k\in \cQ\atop  q_k\ge  q_{k+1}^{1/B}} \{m\in \Z: q_k\le |m|< q_{k+1}\}\bigg).
	\end{eqnarray*}
	Then $\phi^+$ and $\phi^-$ are defined as
	\begin{eqnarray}\label{eqrp}
		\phi^+(t) =\sum_{m\in M_1(B)} \widehat{\phi}(m) e(mt)\quad \mbox{and} \quad
		\phi^-(t) =\sum_{m\in M_2(B)} \widehat{\phi}(m) e(mt).
	\end{eqnarray}
	A keypoint is that $\phi^-$ is expressible as the difference $\Delta_\alpha g_\phi(t):=g_\phi(t+\alpha)-g_\phi(t)$ of another smooth function $g_\phi$ on $\R/\Z$.
	In fact,
	\begin{eqnarray*}
		g_\phi(t) = \sum_{m\in M_2(B)} \widehat{\phi}(m) \frac{e(mt)}{e(m\alpha)-1},
	\end{eqnarray*}
	where the series  on the right side converges absolutely by \cite[Lemma 4.2]{HLW} and by the fact that $\widehat{\phi}(m)\ll_A (1+|m|)^{-A}$ for any $A>0$. Along the same line of argument in the proof (of that lemma), it follows plainly that $g_\phi(t)$ is smooth.
	
	Now we construct the map $R$ by setting
	\begin{eqnarray*}
		R: (t,\overline{g}) \mapsto \left(t,\overline{g} \begin{pmatrix}\begin{smallmatrix}
				1 & g_\varphi(t) & \\ & 1 & g_\eta(t) \\ & & 1\end{smallmatrix}\end{pmatrix} \begin{pmatrix}\begin{smallmatrix} 1 & \phantom{g_\varphi(t)} & c(t) \\ & 1 & \phantom{g_\varphi}\\ & & 1\end{smallmatrix}\end{pmatrix}\right),
	\end{eqnarray*}
	where $c(t)$ will be determined later.
	Following \S\S~\ref{ss3.1}, $R^{-1}$ maps $(l_\alpha(t),\overline{g})$  to
	\begin{eqnarray*}
		\left(l_\alpha(t),\overline{g} \begin{pmatrix}\begin{smallmatrix} 1 & -g_\varphi(l_\alpha(t)) & \\ & 1 & -g_\eta(l_\alpha(t)) \\ & & 1 \end{smallmatrix}\end{pmatrix} \begin{pmatrix}\begin{smallmatrix} 1 & \phantom{-g_\varphi(l_\alpha(t))}& g_\varphi(l_\alpha(t))g_\eta(l_\alpha(t))-c(l_\alpha(t)) \\ & 1 & \phantom{-g_\varphi(l_\alpha(t))} \\ & & 1 \end{smallmatrix}\end{pmatrix}\right),
	\end{eqnarray*}
	where $l_\alpha$ is the translation $l_\alpha(t)=t+\alpha$.
	Thus, letting $S=S_\alpha$ and applying \eqref{S}, the map $R^{-1} SR $ sends $(t,\overline{g})$ to
	\begin{eqnarray*}
		&& R^{-1} S \left(t,\overline{g} \begin{pmatrix}\begin{smallmatrix} 1 \,& g_\varphi(t) & \\ & 1 & g_\eta(t) \\ & & 1 \end{smallmatrix}\end{pmatrix} \begin{pmatrix}\begin{smallmatrix} 1 \, & \phantom{g_\varphi(t)} & c(t) \\ & 1 & \phantom{g_\varphi}\\ & & 1 \end{smallmatrix}\end{pmatrix}\right)\\
		&=& R^{-1}  \left(l_\alpha(t),\overline{g} \begin{pmatrix}\begin{smallmatrix} 1 \, & \varphi(t) +g_\varphi(t) & \\ & 1 & \eta(t)+ g_\eta(t) \\ & & 1 \end{smallmatrix}\end{pmatrix} \begin{pmatrix}\begin{smallmatrix} 1 &\phantom{\varphi(t) +g_\varphi(t)} & \psi(t)+g_\varphi(t) \eta(t)+ c(t) \\ & 1 & \phantom{\varphi(t) +g_\varphi(t)}\\ & & 1 \end{smallmatrix}\end{pmatrix}\right)\\
		&=&
		\left(l_\alpha(t) ,\overline{g} \begin{pmatrix}\begin{smallmatrix} 1 \, & \varphi^+(t)  & \\ & 1 & \eta^+(t) \\ & & 1 \end{smallmatrix}\end{pmatrix} \begin{pmatrix}\begin{smallmatrix} 1 & \phantom{\varphi(t) +g_\varphi(t)}&  \omega(t)- \big(c( l_\alpha(t))-c(t)\big) \\ & 1 & \\ & & 1 \end{smallmatrix}\end{pmatrix}\right),
	\end{eqnarray*}
	where we have used $\phi^+(t)= \phi(t) - \Delta_\alpha g_\phi(t)= \phi(t) - (g_\phi\circ l_\alpha(t)-g_\phi(t))$ and denoted
	\begin{eqnarray*}
		\omega(t)
		&:=&\psi(t)+(g_\varphi \eta)(t) -(\varphi(t)+g_\varphi(t))\cdot g_\eta\circ l_\alpha(t)+ (g_\varphi g_\eta)\circ l_\alpha(t) \\
		&=& \psi(t) +g_\varphi(t) \eta(t)- \varphi^+(t)g_\eta(t+\alpha).
	\end{eqnarray*}
	Note that $\psi, \varphi, g_\varphi, g_\eta$ are smooth functions on $\R/\Z$, so is $\omega$.  Take $c= g_\omega$, then $\Delta_\alpha c(t) = c\circ l_\alpha(t)-c(t) = \omega^-(t)$,
	where $\omega^-(t)$ is the non-resonant part of the function $\omega (t)$.
	Consequently, $T_1:=R^{-1} SR$ satisfies
	\begin{eqnarray}\label{eqT1}
		T_1(t, \overline{g}) &=& \left(l_\alpha(t) ,\overline{g} \begin{pmatrix}\begin{smallmatrix}  1 & \varphi^+(t)  &  \\ & 1 & \eta^+(t) \\ & & 1 \end{smallmatrix}\end{pmatrix}\begin{pmatrix}\begin{smallmatrix}  1 &   & \omega^+(t) \\ & 1 &  \\ & & 1 \end{smallmatrix}\end{pmatrix}\right) \nonumber\\
		&=& (l_\alpha(t) ,\overline{g}h_0(t)w_0(t)), \mbox{ say}.
	\end{eqnarray}
	
	In summary, the map $R$ is constructed as follows:
	\begin{eqnarray*}
		R: (t,\overline{g}) \mapsto \left(t,\overline{g} \begin{pmatrix}\begin{smallmatrix} 1 \, & g_\varphi(t) & g_{\omega}(t)
				\\ & 1 & g_\eta(t) \\ & & 1 \end{smallmatrix}\end{pmatrix}\right).
	\end{eqnarray*}

	\subsection{Formula for the iteration of $T_1$}\label{ss8.2}
	As in \S\S~\ref{ss4.1}, we obtain from \eqref{eqT1} that the $m$-fold composition $T_1^m$ is given by
	\begin{eqnarray}\label{eq8.2}
		T_1^m(t,\overline{g})
		&=&
		\big(l_\alpha^m (t), \overline{g}(h_0h_1\cdots h_{m-1})\cdot (w_0w_1\cdots w_{m-1})(t)\big)\nonumber\\
		&=&
		\left(l_\alpha^m (t), \overline{g} \begin{pmatrix}\begin{smallmatrix} 1 \, & \Phi_m(t)   & \Psi_m(t) \\ & 1 & \xi_m(t)  \\ & & 1 \end{smallmatrix}\end{pmatrix}\right),
	\end{eqnarray}
	where, in this case,
	\begin{eqnarray*}
		\Phi_m(t)&=& \sum_{r=0}^{m-1}\varphi^+\circ l_\alpha^r(t) = \sum_{r=0}^{m-1}\varphi^+(t+r\alpha),
		\\
		\xi_m(t)&=& \sum_{r=0}^{m-1} \eta^+\circ l_\alpha^r (t) = \sum_{r=0}^{m-1}\eta^+(t+r\alpha),
		\\
		\Psi_m(t)&=& \Omega_m(t)  + H_m(t) = \sum_{r=0}^{m-1}\omega^+(t+r\alpha) + \sum_{j=1}^{m-1} \sum_{r=0}^{j-1} \varphi^+(t+r\alpha) \cdot\eta^+(t+j\alpha).
	\end{eqnarray*}
	
	

	\section{Proof of Theorem~\ref{prop}}
	
	Given any small $\varepsilon>0$ and  $L>\varepsilon^{-1}$ a large enough number. Let $\tau\in (0,1)$ be any arbitrarily small number and define $B=6\tau^{-1}+1$. 
	
	To finish the proof of Theorem~\ref{prop}, 
	we will show in \S\S~\ref{QBinfty} and \S\S~\ref{QBfinite} that the measure complexity of $(\mathbb{T}\times\Gamma\backslash G, S_{\alpha}, \rho)$ fulfills \eqref{sp} for infinite or finite $\mathcal{Q}_B$ defined in \eqref{eqqB}.
	
	\subsection{Case of $\#\cQ_B = \infty$}\label{QBinfty}
	Based on the explanation in Section~\ref{s8}, our task is to show the following:
	for all large $n_k=q_k^{B-1}\in \cQ_B$, for any adjacent pair $(t,h), (t^*,h^*)\in F_\varepsilon (k)$ so that   $|t-t^*|\le \varepsilon/ (q_k^2 L)$ and $\max(|x-x^*|, |y-y^*|,|z-z^*|) \le 1/(q_k L)$ if $h=\begin{pmatrix}\begin{smallmatrix} 1 & y & z \\ & 1 & x\\ & & 1 \end{smallmatrix}\end{pmatrix}$ and $h^*=\begin{pmatrix}\begin{smallmatrix} 1 & y^* & z^* \\ & 1 & x^*\\ & & 1\end{smallmatrix}\end{pmatrix}$, we have
	\begin{eqnarray}\label{eqdbar}
		\overline{d}_{n_k}((t,h), (t^*,h^*))<\varepsilon ,
	\end{eqnarray}
	where for $T_1$ defined as in \eqref{eqT1},
	\begin{eqnarray}\label{eqdbar2}
		\overline{d}_{n_k} ((t,h), (t^*,h^*)):= \frac1{n_k} \sum_{m=0}^{n_k-1} d(T_1^m(t,h), T_1^m(t^*,h^*)).
	\end{eqnarray}
	This will lead to \eqref{sp} by \eqref{eq8.1} because for any $\varepsilon>0$,
	$$
	\liminf_{k\to \infty} \frac{\#F_\varepsilon (k)}{n_k^\tau}
	\le \liminf_{k\to \infty} \frac{\varepsilon^{-1} q_k^5 L^4}{q_k^{(B-1)\tau}}
	=\liminf_{k\to \infty} \frac{\varepsilon^{-1}L^4}{q_k}  =0, \quad \mbox{$\forall$ $\tau >0$}.
	$$
	
	To verify \eqref{eqdbar}, we firstly compute $d(T_1^m (t,h), T_1^m(t^*,h^*))$ using \S\S~\ref{ss3.2} with 
	$$
	g=h, \ g^*=h^*, \ Y=   
	\begin{pmatrix}\begin{smallmatrix}
			1 & \Phi_m(t)   & \Psi_m(t) \\ & 1 & \xi_m(t)  \\ & & 1
	\end{smallmatrix} \end{pmatrix} , \ 
	Y^* =   \begin{pmatrix}\begin{smallmatrix} 1 & \Phi_m(t^*)   & \Psi_m(t^*) \\ & 1 & \xi_m(t^*)  \\ & & 1\end{smallmatrix}\end{pmatrix} .
	$$
	Indeed,
	\begin{eqnarray}\label{eqdm}
		&& d(T_1^m (t,\overline{h}), T_1^m(t^*,\overline{h^*}))\nonumber\\
		&\le & \|t-t^*\| + d_{\Gamma\backslash G} (\overline{hY}, \overline{h^*Y^*})\nonumber\\
		&\le& \|t-t^*\| +  (1+|\Phi_m(t)|+ |y|) |x-x^*|+  (1+|\xi_m(t)|) |y-y^*| + |z-z^*| \nonumber\\
		& & +\big(1+|\Phi_m(t)|\big)|\xi_m(t)-\xi_m(t^*)|+|\Phi_m(t)-\Phi_m(t^*)| +  \|\Psi_m(t)-\Psi_m(t^*)\|\nonumber\\ 
		&\le & \frac{\varepsilon}{q_k^2 L} + \frac1{q_k L} (4+|\xi_m(t)|+|\Phi_m(t)|)\nonumber\\
		&& +\big(1+|\Phi_m(t)|\big)|\xi_m(t)-\xi_m(t^*)|+|\Phi_m(t)-\Phi_m(t^*)| +  \|\Psi_m(t)-\Psi_m(t^*)\|.
	\end{eqnarray}
	
	By (the argument in) \cite[Lemma 4.3]{HLW}, we obtain that  for any smooth function $\phi$ on $\R/\Z$,  there is a constant $c_B>0$ such that for any $q\in \cQ$ with $q^{B-1}\in \cQ_B$,
	\begin{eqnarray}\label{eqCS}
		\Big|\frac1{q} \sum_{j=0}^{q} \phi\circ l_\alpha^j(t) - \widehat{\phi}(0)\Big| \le c_B q^{-B}.
	\end{eqnarray}
	Consequently, since $\widehat{\varphi^+}(0) = \widehat{\varphi}(0)=0$ and $\widehat{\eta^+}(0)= \widehat{\eta}(0)=0$,  for $0\le m< n_k$ we have
	\begin{eqnarray}\label{eqPx}
		|\Phi_m(t)| + |\xi_m(t)| 
		&\le&  m(|\widehat{\varphi^+}(0)|+ |\widehat{\eta^+}(0)|)+ 2c_B \frac{n_k}{q_k} q_k^{1-B} + q_k (|\varphi^+|_\infty+ |\eta^+|_\infty) \nonumber\\
		&\ll& \frac1{q_k}+  q_k (|\varphi^+|_\infty+ |\eta^+|_\infty)
	\end{eqnarray}
	and by \eqref{eqdm},
	\begin{eqnarray}\label{eqdm2}
		d(T_1^m (t,\overline{h}), T_1^m(t^*,\overline{h^*})) 
		&\ll & \frac1L(1+|\varphi^+|_\infty+ |\eta^+|_\infty)
		+\big( 1+ |\Phi_m(t)|\big)|\xi_m(t)-\xi_m(t^*)|  \nonumber\\
		&& 
		\ \ +|\Phi_m(t)-\Phi_m(t^*)| +  |\Psi_m(t)-\Psi_m(t^*)|.
	\end{eqnarray}
	The first summand on the right-side is $< \varepsilon$ when a sufficiently large $L$ is chosen.
	It remains to evaluate the remaining three summands in \eqref{eqdm2}. 
	
	By \eqref{eqCS}, we see that
	\begin{eqnarray}\label{eqphi}
		\Big|\sum_{j=0}^{m-1} \phi\circ l_\alpha^j(t) - \sum_{j=0}^{m-1} \phi\circ l_\alpha^j(t^*) \Big|
		\le 2 c_B q_k^{-1} + q_k |\phi'|_\infty |t-t^*|,
	\end{eqnarray}
	which will be applied to $\Phi_m(t)-\Phi_m(t^*)$ and $\Omega_m(t)-\Omega_m(t^*)$. (Recall $\Psi_m= \Omega_m + H_m$.) However for our purpose, the estimate in \eqref{eqphi} for $\xi_m(t)-\xi_m(t^*)$ is not enough in view of \eqref{eqPx}, and as well, it is not applicable to $H_m(t)-H_m(t^*)$. Therefore we need another auxiliary result, which is Lemma~\ref{lem1} in the final section. Applying Lemma~\ref{lem1} with $B=6\tau^{-1}+1$ to $H_m$ and \eqref{eqphi} to $\Omega_m$, we have, as $\Psi_m= \Omega_m+H_m$,
	$$
	\Psi_m(t)-\Psi_m(t^*)\ll \frac1{q_k} + (q_k^2 +q_k|{\omega^+}'|_\infty)|t-t^*|\ll q_k^{-1} + q_k^2|t-t^*|.
	$$
	The same upper bound applies to $\big( 1+ |\Phi_m(t)|\big)|\xi_m(t)-\xi_m(t^*)|$ by Lemma~\ref{lem1}.
	Inserting them and the estimates from \eqref{eqphi} with $\phi = \varphi^+$ and $\eta^+$
	into \eqref{eqdm2}, we get that
	\begin{eqnarray*}
		d(T_1^m (t,\overline{h}), T_1^m(t^*,\overline{h^*}))  
		&\ll & \frac1L(1+|\varphi^+|_\infty+ |\eta^+|_\infty) + q_k^{-1} + q_k^2 |t-t^*| \ll \varepsilon.
	\end{eqnarray*}
	In view of \eqref{eqdm} and \eqref{eqdbar2}, we complete the establishment of \eqref{eqdbar}.

	\subsection{Case of finite $\#\cQ_B$}\label{QBfinite}
	If $\#\cQ_B$ is finite, then $\Z\setminus M_2(B) = M_1(B)$ is also finite and hence by \cite[Lemma 4.2]{HLW}, 
	\begin{eqnarray*}
		\widetilde{g}_\varphi(t) = \sum_{m\neq0} \widehat{\varphi}(m) \frac{e(mt)}{e(m\alpha)-1},
		\qquad
		\widetilde{g}_\eta(t) = \sum_{m\neq0} \widehat{\eta}(m) \frac{e(mt)}{e(m\alpha)-1}
	\end{eqnarray*}
	are absolutely convergent, and as functions in $t$, $\widetilde{g}_\varphi(t)$ and 
	$\widetilde{g}_\eta(t)$ are continuous and periodic with period $1$.
	Under the conditions $\widehat{\varphi}(0)=\widehat{\eta}(0)=0$, the process in \S\S~\ref{ss8.1} and \S\S~\ref{ss8.2} will transform the flow to an easily manageable measurable isomorphic flow. 
	
	In fact, now we have 
	\begin{eqnarray*}
		\varphi(t)= \widetilde{g}_\varphi\circ l_\alpha(t) - \widetilde{g}_\varphi(t),
		\qquad
		\eta(t)= \widetilde{g}_\eta\circ l_\alpha(t) - \widetilde{g}_\eta(t),
	\end{eqnarray*}
	for which $\widehat{\varphi}(0)=\widehat{\eta}(0)=0$ is applied. Set
	\begin{eqnarray*}
		\varpi(t):=\psi(t)+(\widetilde{g}_\varphi \eta)(t) -(\varphi(t)+\widetilde{g}_\varphi(t))\cdot \widetilde{g}_\eta\circ l_\alpha(t)+ (\widetilde{g}_\varphi \widetilde{g}_\eta)\circ l_\alpha(t) ,
	\end{eqnarray*}
	which is smooth and of period 1 as $\psi, \varphi, \widetilde{g}_\varphi, \widetilde{g}_\eta$ are smooth functions on $\R/\Z$.  Then,
	\begin{eqnarray*}
		\varpi(t)=\widehat{\varpi}(0)  + \widetilde{g}_\varpi\circ l_\alpha(t) - \widetilde{g}_\varpi(t)
	\end{eqnarray*}
	and the inverse of the map
	\begin{eqnarray*}
		\widetilde{R}: (t,\overline{g}) \mapsto
		\left(t,\overline{g} \begin{pmatrix}\begin{smallmatrix} 1 \, & \widetilde{g}_\varphi(t) & \\ & 1 & \widetilde{g}_\eta(t) \\ & & 1 \end{smallmatrix}\end{pmatrix} \begin{pmatrix}\begin{smallmatrix} 1 &  \phantom{\widetilde{g}_\varphi(t)} & \widetilde{g}_\varpi(t) \\ & 1 &  \phantom{\widetilde{g}_\varphi}\\ & & 1 \end{smallmatrix}\end{pmatrix}\right)
	\end{eqnarray*}
	sends $(t,\overline{g})$ to
	\begin{eqnarray*}
		\left(t,\overline{g} \begin{pmatrix}\begin{smallmatrix} 1 \, & -\widetilde{g}_\varphi(t) & \\ & 1 & -\widetilde{g}_\eta(t) \\ & & 1 \end{smallmatrix}\end{pmatrix} \begin{pmatrix}\begin{smallmatrix} 1 & \phantom{-\widetilde{g}_\varphi(t)}& \widetilde{g}_\varphi(t)\widetilde{g}_\eta(t)-\widetilde{g}_\varpi(t) \\ & 1 & \phantom{\widetilde{g}_\varphi}\\ & & 1 \end{smallmatrix}\end{pmatrix}\right).
	\end{eqnarray*}
	Thus, $\tilde{T}_1:=\widetilde{R}^{-1} S\widetilde{R} $ maps $(t,\overline{g})$ to
	\begin{eqnarray*}
		&& \widetilde{R}^{-1} S \left(t,\overline{g} \begin{pmatrix}\begin{smallmatrix}  1 \, & \widetilde{g}_\varphi(t) & \\ & 1 & \widetilde{g}_\eta(t) \\ & & 1 \end{smallmatrix}\end{pmatrix} \begin{pmatrix}\begin{smallmatrix}  1 & \phantom{\widetilde{g}_\varphi(t)} & \widetilde{g}_\varpi(t) \\ & 1 &  \phantom{\widetilde{g}_\varphi}\\ & & 1 \end{smallmatrix}\end{pmatrix}\right)\\
		&=& \widetilde{R}^{-1}  \left(l_\alpha(t),\overline{g} \begin{pmatrix}\begin{smallmatrix}  1 \, & \varphi(t) +\widetilde{g}_\varphi(t) & \\ & 1 & \eta(t)+ \widetilde{g}_\eta(t) \\ & & 1 \end{smallmatrix}\end{pmatrix} \begin{pmatrix}\begin{smallmatrix}  1 &  \phantom{\widetilde{g}_\varphi(t)} & \psi(t)+\widetilde{g}_\varphi(t) \eta(t)+ \widetilde{g}_\varpi(t) \\ & 1 &  \phantom{\widetilde{g}_\varphi} \\ & & 1 \end{smallmatrix}\end{pmatrix}\right)\\
		&=&
		\left(l_\alpha(t) ,\overline{g}
		\begin{pmatrix}\begin{smallmatrix}  1 &  \phantom{\widetilde{g}_\varphi(t)} &  \varpi(t)- \big(\widetilde{g}_\varpi\circ l_\alpha(t)-\widetilde{g}_\varpi(t)\big) \\ & 1 & \\ & & 1 \end{smallmatrix}\end{pmatrix}\right)\\
		&=& \left(l_\alpha(t) ,\overline{g}
		\begin{pmatrix}\begin{smallmatrix} 
				1 &  \phantom{\widetilde{g}_\varphi}  &  \widehat{\varpi}(0)
				\\ & 1 &
				\\ & & 1
		\end{smallmatrix}\end{pmatrix}
		\right),
	\end{eqnarray*}
	where $S=S_\alpha$.
	By  Proposition~\ref{prop2.3}, 
	the measure complexity of $(\mathbb{T}\times\Gamma\backslash G, S, \rho)$ is sub-polynomial, i.e. satisfying \eqref{sp}, if and only if the measure complexity of $(\mathbb{T}\times\Gamma\backslash G, \tilde{T}_1, v)$ with $v=\rho\circ \tilde{R}$  is sub-polynomial. From \S\S~\ref{ss3.2}, the metric $d$ is seen to be invariant under $\tilde{T}_1$, so for any $n\geq1$ and $\varepsilon>0$, 
	$$
	s_{n}(\mathbb{T}\times\Gamma\backslash G, \tilde{T}_1, d, v, \varepsilon)=s_{1}(\mathbb{T}\times\Gamma\backslash G, \tilde{T}_1, d, v, \varepsilon).
	$$
	Since $\mathbb{T}\times\Gamma\backslash G$ is compact, we have $s_{1}(\mathbb{T}\times\Gamma\backslash G, \tilde{T}_1, d, v, \varepsilon)<\infty$ and consequently
	$$
	\lim_{n\rightarrow\infty}\frac {s_{n}(\mathbb{T}\times\Gamma\backslash G, \tilde{T}_1, d, v, \varepsilon)}{n^\tau}=\lim_{n\rightarrow\infty}\frac {s_{1}(\mathbb{T}\times\Gamma\backslash G, \tilde{T}_1, d, v, \varepsilon)}{n^\tau}=0.
	$$
	Hence the measure complexity of $(\mathbb{T}\times\Gamma\backslash G, S, \rho)$ satisfies \eqref{sp} for finite $\#\cQ_B$ as well.

	\section{A technical lemma}
	\begin{lemma}\label{lem1} Let $B>2$. There are constants $c_B',c_B''>0$, independent of $q_k$, such that
		\begin{eqnarray*}
			|H_m(t^*)-H_m(t)| &\le& c_B'\big( \frac1{q_k^2} + q_k^2 |t-t^*|\big),\\
			|\xi_m(t^*)-\xi_m(t)| &\le& c_B''\big( \frac1{q_k^2} + q_k |t-t^*|\big)
		\end{eqnarray*}
		for all  $1\le m\le n_k (= q_k^{B-1}\in \mathcal{Q}_B)$. 
	\end{lemma}
	\begin{proof}
		By the definitions of $H_m(t)$ (below \eqref{eq8.2}) and the resonant parts \eqref{eqrp}, we have 
		\begin{eqnarray*}
			H_m(t) &=& \sum_{l=1}^{m-1} \sum_{r=0}^{l-1} \varphi^+(t+r\alpha) \cdot\eta^+(t+l\alpha)\\
			&=& \sum_{\substack{q_a,q_c\in \cQ\\ 1< q_a < q_{a+1}^{1/B}\\ 1< q_c < q_{c+1}^{1/B}}} \sum_{\substack{q_a\le |u|< q_{a+1}, \, q_a|u\\ q_c\le |v|< q_{c+1}, \, q_c|v}}
			\widehat{\varphi}(u) \widehat{\eta}(v)
			\sum_{l=1}^{m-1} \sum_{r=0}^{l-1} e((ur+vl)\alpha) e((u+v)t)
		\end{eqnarray*}
		under the condition $\widehat{\varphi}(0)=\widehat{\eta}(0)=0$.
		The constant term (i.e. independent of $t$) of $H_m(t)$ comes from summands with $u+v =0$, and it will not appear in $H_m(t^*)-H_m(t)$ due to cancellation. Hence we can add the constraint $u+v\neq 0$ when estimating the multiple sum.  
		As $m\le n_k= q_k^{B-1}$, we express $m=a_kq_k+{\rm b}$, where $0\le {\rm b}< q_k$, and invoke the rapid decay  $\widehat{\varphi}(u)\ll (1+|u|)^{-2B-2}$ and $ \widehat{\eta}(v)\ll (1+|v|)^{-2B-2} $ to derive that
		\begin{eqnarray}\label{eqHtail}
			\mbox{ }\quad \sum_{\substack{q_a,q_c\in \cQ\\ a \ {\rm or } \ c \ge k}} + \sum_{\substack{q_a,q_c\in \cQ\\ a ,\, c < k}} \sum_{|u+v|\ge q_k}
			&\ll& 
			\sum_{\substack{q\in \cQ\\ q\ge q_k}} \sum_{u \ge q} |u|^{-2B-2} m^2
			+
			\sum_{\substack{q\in \cQ}} \sum_{u \gg \max(q_k, q)} |u|^{-2B-2} m^2 \\
			&\ll& q_k^{-2}, \nonumber  
		\end{eqnarray}
		the $m^2$ in which comes from the trivial bound on the double summation over $l,r$.
		We denote the non-constant remnant by
		$$
		H_m^\circ (t) :=
		\sum_{\substack{1< q_a,q_c\in \cQ\\ q_a < q_{a+1}^{1/B}\le q_k^{1/B}\\ q_c < q_{c+1}^{1/B}\le q_k^{1/B}}} \sum_{\substack{q_a\le |u|< q_{a+1}, \, q_a|u\\ q_c\le |v|< q_{c+1}, \, q_c|v\atop  0< |u+v|< q_k}}
		\widehat{\varphi}(u) \widehat{\eta}(v)  e((u+v)t)
		\sum_{l=1}^{m-1} \sum_{r=0}^{l-1} e((ur+vl)\alpha).
		$$
		
		Now we split the double summation over $(l,r)$ inside $H_m^\circ (t)$  into
		\begin{eqnarray*}
			\sum_{{\rm a}=0}^{a_k-1} \sum_{l={\rm a} q_k+1}^{({\rm a}+1)q_k} \sum_{r=0}^{l-1} + \sum_{l=a_kq_k+1}^{a_kq_k+{\rm b}} \sum_{r=0}^{l-1}=: \sum_{{\rm a} =0}^{a_k-1} \mbox{$\sum_{\rm a}$} + \mbox{$\sum_{\rm b}$}, \mbox{ say}.
		\end{eqnarray*}
		A direct calculation gives that
		\begin{eqnarray*}
			\mbox{$\sum_{\rm a}$} &=& \sum_{l={\rm a} q_k+1}^{({\rm a}+1)q_k} \sum_{r=0}^{l-1}e((ur+vl)\alpha)\\
			&=& \frac{e(({\rm a}q_k+1)(u+v)\alpha)\big(e(q_k(u+v)\alpha)-1\big)}{\big(e(u\alpha)-1\big)\big(e((u+v)\alpha)-1\big)}
			- \frac{e(({\rm a}q_k+1)v\alpha)\big(e(q_k v\alpha)-1\big)}{\big(e(u\alpha)-1\big)\big(e(v\alpha)-1\big)}, \end{eqnarray*}
		so
		\begin{eqnarray*}
			\big|\mbox{$\sum_{\rm a}$} \big|
			\le  \frac{\big|e(q_k(u+v)\alpha)-1\big|}{\big|e(u\alpha)-1\big|\cdot \big|e((u+v)\alpha)-1\big|}   + \frac{\big|e(q_k v\alpha)-1\big|}{\big|e(u\alpha)-1\big| \cdot \big|e(v\alpha)-1\big|}.
		\end{eqnarray*}
		
		We make use of the ideas and arguments in \cite{HLW}:
		If $1<q_a\le |u|< q_{a+1}$ and $q_a \big| |u|$, then letting $|u|=dq_a$ ($<q_{a+1}$), we obtain, by \eqref{p2}, $(2q_{a+1})^{-1}< \|q_a\alpha\| \le \|u\alpha\|=\|dq_a\alpha\| \le d\| q_a\alpha\|<d/q_{a+1}<1/q_a\le \frac12$, implying $ \|dq_a\alpha\|= d\|q_a \alpha\|$ (from $d\|q_a\alpha\|<\frac12$) and 
		\begin{eqnarray}\label{equa}
			\frac1{\big|e(u\alpha)-1\big|}&\ll& \|u\alpha \|^{-1} = (d\|q_a \alpha\|)^{-1}\le \frac{2q_{a+1}}d= \frac{2q_aq_{a+1}}{|u|}\\
			&\ll& \frac{q_k^2}{|u|} \quad \mbox{  when  $a<k$.}\nonumber
		\end{eqnarray}
		Easily, with  \eqref{p2},
		$$
		\big|e(q_k(u+v)\alpha)-1\big|\ll |u+v|\|q_k\alpha\| < \frac{|u+v|}{q_{k+1}}.
		$$
		To take $|e((u+v)\alpha)-1|^{-1}$ into account, we separate into two cases:
		\begin{itemize}
			\item
			If $\|(u+v)\alpha\|\ge 1/(2|u+v|)$, then we get immediately
			\begin{eqnarray*}
				\frac{\big|e(q_k(u+v)\alpha)-1\big|}{\big|e(u\alpha)-1\big|\cdot \big|e((u+v)\alpha)-1\big|}
				\ll \frac{q_k^2|u+v|^2}{|u|} \frac1{q_{k+1}} < \frac{q_k^4}{|u|} \frac1{q_{k+1}}.
			\end{eqnarray*}
			\item
			If $ 0< \|(u+v)\alpha\|< 1/(2|u+v|)$, then  $u+v\neq 0$ and $l/(u+v)=l_i/q_i$ is a convergent for some $l$  (and so $q_i\in \cQ$), by \eqref{p1}.
			In particular, $q_i\big| u+v$  and $i<k$ as $|u+v|< q_k$. By the same argument as in \eqref{equa}, we have $|e((u+v)\alpha)-1|^{-1} \ll q_k^2/|u+v|$. Thus,
			$$
			\frac{\big|e(q_k(u+v)\alpha)-1\big|}{\big|e(u\alpha)-1\big|\cdot \big|e((u+v))\alpha)-1\big|}
			\ll \frac{q_k^4}{|u|} \frac1{q_{k+1}}.
			$$
		\end{itemize}
		Also, $|e(v\alpha)-1|^{-1} \ll q_k^2/|v|$ (as in \eqref{equa}) and $|e(vq_k\alpha)-1| \ll |v|/q_{k+1}$. We infer that
		$$
		\big|\mbox{$\sum_{\rm a}$} \big|
		\ll \frac{q_k^4}{|u|} \frac1{q_{k+1}}.
		$$
		
		Similarly,
		\begin{eqnarray*}
			\big|\mbox{$\sum_{\rm b}$}\big| &=& \bigg|\bigg(\sum_{l=a_k q_k+1}^{a_kq_k+{\rm b}} \sum_{r=0}^{a_kq_k-1}+\sum_{l=a_k q_k+1}^{a_kq_k+{\rm b}} \sum_{r=a_kq_k}^{l-1}\bigg)e((ur+vl)\alpha)\bigg|\\
			&\le&  \frac{\big|e(a_kq_k u\alpha)-1\big|}{\big|e(u\alpha)-1\big|}{\rm b} +q_k^2
			\ll \frac{q_k^2}{|u|}\frac{a_k|u|}{q_{k+1}}q_k+q_k^2\ll q_k^2,
		\end{eqnarray*}
		for $a_kq_k\le n_k= q_k^{B-1}$ and $q_{k+1} > q_k^B$, in view of \eqref{eqqB}.
		
		From the above estimates and the fact $e((u+v)t)-e((u+v)t^*)\ll |u+v| |t-t^*|$, we conclude
		\begin{eqnarray*}
			H_m^\circ(t) - H_m^\circ(t^*)
			&\ll&  \sum_{1<|u|, |v|\le q_k} (|u||v|)^{-2B-2}|u+v| |t-t^*| \Big(\frac{a_k q_k^{4}}{|u|q_{k+1}}+q_k^2\Big)\\
			&\ll& q_k^2 |t-t^*|.
		\end{eqnarray*}
		Together with \eqref{eqHtail}, the case of $H_m$ is done.

		The proof for $\xi_m$ follows in the same vein. By \eqref{eqrp}, we have
		\begin{eqnarray*}
			\xi_m(t) &=& \sum_{r=0}^{m-1}  \eta^+(t+r\alpha)= \sum_{\substack{q_c\in \cQ\\ 1< q_c < q_{c+1}^{1/B}}} \sum_{\substack{q_c\le |v|< q_{c+1} \\ q_c|v}}
			\widehat{\eta}(v) e(vt)
			\sum_{r=0}^{m-1} e(vr\alpha) \\
			&=:& \xi_{m,1}(t) + \xi_{m,2}(t), \mbox{ say,}
		\end{eqnarray*}
		by splitting the sum over $q_c$ into $c<k$ or $c\ge k$ respectively. Thus
		\begin{eqnarray*}
			\xi_{m,2}(t) &\ll& \sum_{\substack{q_c\in \cQ\\ q_k\le q_c < q_{c+1}^{1/B}}}
			\sum_{\substack{q_c\le |v|< q_{c+1} \\ q_c|v}} |v|^{-2B-2} m
			\ll
			\sum_{v\ge q_k} v^{-2B-2} q_k^{B-1} \ll q_k^{-2},
		\end{eqnarray*}
		and using the fact ${|e(q_kv\alpha)-1|}/{|e(v\alpha)-1|}\ll q_k^2/q_{k+1}$ for $q_c\le |v|< q_{c+1}$,  $q_c|v$ and $c<k$,
		\begin{eqnarray*}
			\xi_{m,1}(t)-\xi_{m,1}(t^*)
			&=& \sum_{\substack{1< q_c\in \cQ\\ q_c < q_{c+1}^{1/B}\le q_k^{1/B}}}
			\sum_{\substack{q_c\le |v|< q_{c+1} \\ q_c|v}}\widehat{\eta}(v) \big(e(vt)-e(vt^*)\big)
			\sum_{r=0}^{m-1} e(vr\alpha) \\
			&\ll& \sum_{\substack{1< q_c\in \cQ\\ q_c < q_{c+1}^{1/B}\le q_k^{1/B}}}
			\sum_{\substack{q_c\le |v|< q_{c+1} \\ q_c|v}} |v|^{-2B-1} |t-t^*| \Big(\frac{n_k}{q_k} \frac{|e(q_kv\alpha)-1|}{|e(v\alpha)-1|} + q_k\Big)\\
			&\ll&|t-t^*|  q_k \sum_{v\ge 1} v^{-2B-1}\ll q_k|t-t^*| .
		\end{eqnarray*}
		The proof is complete.
	\end{proof}
	
	\subsection*{Funding} Lau is supported by GRF (No. 17317822) of  the Research Grants Council of Hong Kong and the National Natural Science Foundation of China (NSFC No. 12271458).

	\bibliographystyle{plainnat}
	
\end{document}